\documentclass[12pt,reqno]{amsart}

\usepackage{caption,amsmath,amssymb,amsfonts,amsthm,latexsym,graphicx,multirow,color,hyperref,enumerate}

\newcommand\A{\mathrm{A}} \newcommand\AGL{\mathrm{AGL}} \newcommand\AGaL{\mathrm{A\Gamma L}} \newcommand\Alt{\mathrm{Alt}} \newcommand\ASL{\mathrm{ASL}} \newcommand\ASiL{\mathrm{A\Sigma L}} \newcommand\Aut{\mathrm{Aut}}
 \newcommand\bbF{\mathbb{F}} 
\newcommand\C{\mathrm{C}} \newcommand\Cen{\mathbf{C}}   \newcommand\calO{\mathcal{O}}  \newcommand\Cay{\mathrm{Cay}}  \newcommand\Cos{\mathrm{Cos}}

  \newcommand\Fix{\mathrm{Fix}}
\newcommand\G{\mathrm{G}} \newcommand\GaL{\mathrm{\Gamma L}}    \newcommand\GL{\mathrm{GL}} \newcommand\GO{\mathrm{O}}

\newcommand\K{\mathsf{K}}

\newcommand\M{\mathrm{M}} \newcommand\magma{{\sc Magma}}

\newcommand\Nor{\mathbf{N}}
\newcommand\Out{\mathrm{Out}}
  \newcommand\PGL{\mathrm{PGL}} \newcommand\PGaL{\mathrm{P\Gamma L}}    \newcommand\POm{\mathrm{P\Omega}} \newcommand\PSL{\mathrm{PSL}}    \newcommand\PSp{\mathrm{PSp}} \newcommand\PSU{\mathrm{PSU}}
\newcommand\Q{\mathrm{Q}}
  \newcommand\Ree{\mathrm{Ree}} 
 \newcommand\SL{\mathrm{SL}}  \newcommand\Soc{\mathrm{Soc}} \newcommand\Sp{\mathrm{Sp}}    \newcommand\Sy{\mathrm{S}} \newcommand\Sym{\mathrm{Sym}} \newcommand\Sz{\mathrm{Sz}}

\newcommand\val{\mathsf{Val}}
\newcommand\Z{\mathbf{Z}} \newcommand\ZZ{\mathrm{C}}

\newtheorem{theorem}{Theorem}[section]
\newtheorem{lemma}[theorem]{Lemma}
\newtheorem{proposition}[theorem]{Proposition}
\newtheorem{corollary}[theorem]{Corollary}

\theoremstyle{definition}

\newtheorem{construction}[theorem]{Construction}

\def\l{\langle} \def\r{\rangle}

\def\Ga{\Gamma}   \def\Ome{\Omega}

\def\a{\alpha} \def\b{\beta}  \def\s{\sigma}

\begin{document}

\title[2-Arc-transitive alternating Cayley graphs]
{2-Arc-transitive Cayley graphs on alternating groups}

\author[Pan]{Jiangmin Pan}
\address{J. M. Pan\\
School of Statistics and Mathematics\\
Yunnan University of Finance and Economics\\
Kunming \\
P. R. China}
\email{jmpan@ynu.edu.cn}

\author[Xia]{Binzhou Xia}
\address{B. Z. Xia\\
School of Mathematics and Statistics\\
The University of Melbourne\\
Parkville, VIC 3010\\
Australia}
\email{binzhoux@unimelb.edu.au}

\author[Yin]{Fugang Yin}
\address{F. G. Yin\\
Department of Mathematics\\
Beijing Jiaotong University\\
Beijing 100044, P. R. China}
\email{18118010@bjtu.edu.cn}


\begin{abstract}
An interesting fact is that most of the known connected $2$-arc-transitive nonnormal Cayley graphs of small valency on finite simple groups are $(\mathrm{A}_{n+1},2)$-arc-transitive Cayley graphs on $\mathrm{A}_n$. This motivates the study of $2$-arc-transitive Cayley graphs on $\mathrm{A}_n$ for arbitrary valency. In this paper, we characterize the automorphism groups of such graphs.  In particular, we show that for a non-complete $(G,2)$-arc-transitive Cayley graph on $\mathrm{A}_n$ with $G$ almost simple, the socle of $G$ is either $\mathrm{A}_{n+1}$ or $\mathrm{A}_{n+2}$. We also construct the first infinite family of $(\mathrm{A}_{n+2},2)$-arc-transitive Cayley graphs on $\mathrm{A}_n$.


\end{abstract}

\maketitle

\section{Introduction}

All graphs considered in this paper are finite, simple and undirected. 
For a positive integer $s$, an \emph{$s$-arc} of a graph is a $(s+1)$-tuple of vertices $(v_0,v_1,...,v_s)$ where $v_i$ is adjacent to $v_{i+1}$ for $0\leq i\leq s-1$ and $v_{i-1}\neq v_{i+1}$ for $1\leq i\leq s-1$. 
Let $\Ga$ be a graph and let $G$ be a subgroup of the full automorphism group $\Aut(\Ga)$ of $\Ga$. 
The graph $\Ga$ is said to be \emph{$(G,s)$-arc-transitive} if $G$ is transitive on the set of $s$-arcs. 
For short, we say that $\Ga$ is \emph{$s$-arc-transitive} if it is $(\Aut(\Ga),s)$-arc-transitive, and \emph{arc-transitive} if it is $(\Aut(\Ga),1)$-arc-transitive.
Note that for regular graphs $s$-arc-transitivity with $s\ge 2$ implies $(s-1)$-arc-transitivity.
Tutte~\cite{Tutte1947} in 1947 proved that there exist no finite $s$-arc-transitive cubic graphs for $s \geq 6$.
Since this remarkable result, $s$-arc-transitive graphs have attracted considerable attention in the literature.

Let $H$ be a group and let $S$ be an inverse-closed nonempty subset of $H\setminus\{1\}$. The \emph{Cayley graph} $\Cay(H,S)$ on $H$ with \emph{connection set} $S$ is defined to be the graph with vertex set $H$ such that $x,y\in H$ are adjacent if and only if $yx^{-1}\in S$.

Some special classes of  $2$-arc-transitive Cayley graphs  have been  classified or characterized  in the literature. See  Alspach, Conder, Maru\v{s}i\v{c} and Xu \cite{ACMX96} for $2$-arc-transitive Cayley graphs on cyclic groups, Ivanov and Praeger  \cite {IP} and Li and Pan  \cite{LP08} for $2$-arc-transitive Cayley graphs on abelian groups, Maru\v{s}i\v{c} and Du \cite{DMM08} for $2$-arc-transitive Cayley graphs on dihedral groups,  and the recent work of Li and Xia \cite{LX20+} on $2$-arc-transitive Cayley graphs on solvable groups.

For a Cayley graph $\Cay(H,S)$, denote by $R_H(H)$ the subgroup of $\Sym(H)$ induced by the right multiplication of $H$ on itself. Clearly, $R_H(H)$ is a subgroup of $\Aut(\Cay(H,S))$. The Cayley graph $\Cay(H,S)$ is said to be \emph{normal} if $R_H(H)$ is a normal subgroup of $\Aut(\Cay(H,S))$, and is said to be \emph{nonnormal} otherwise. Nonnormal Cayley graphs are believed to be rare, and it is conjectured by Xu \cite{Xu98} that almost all connected Cayley graphs are normal. Moreover, it is well known that if $\Cay(H,S)$ is a normal Cayley graph on $H$, then $\Aut(\Cay(H,S))=R_H(H){:}\Aut(H,S)$ where 
\[
\Aut(H,S):=\{\s\in\Aut(H)\mid S^{\s}=S\}, 
\]
see Godsil \cite[Lemma 2.1]{G81}.
Hence the symmetries of normal Cayley graphs can be well understood, and consequently more attention in the literature are on nonnormal ones
while investigating Cayley graphs.

The Jordan-H\" older theorem shows that all finite groups can be viewed as group extensions of simple groups (namely their compositor factors).
This fundamental theorem increases the interest of the study of Cayley graphs on simple groups.  
Fang, Praeger and Wang \cite{FPW02} gave a general description of the structures of automorphism groups
of nonnormal Cayley graphs on nonabelian simple groups,
and $s$-arc-transitive nonnormal Cayley graphs on nonabelian simple groups of certain valencies have been  explicitly characterized.
Let $\Ga$ be a connected nonnormal $s$-arc-transitive Cayley graph on a
nonabelian simple group.
Based on a so called `dual action' given in Li \cite{Li96}, Xu, Fang, Wang and Xu \cite{XFWX05,XFWX07}
proved that if the valency $\val(\Ga)=3$ then $\Ga$ is one of two $(\A_{48},5)$-arc-transitive Cayley graphs on $\A_{47}$.
Du and Feng \cite{DF19} proved that if $\val(\Ga)=4$ and $s\ge 2$, then $\Ga$ is an $\A_{n+1}$-arc-transitive Cayley graph on $\A_n$
for $n\in\{24,36,72,144,532,14364\}$ with a unique exception.
Similarly, Du, Feng and Zhou \cite{DFZ17} showed
that if $\val(\Ga)=5$ then $\Ga$ is an $\A_{n+1}$-arc-transitive Cayley graph
on $\A_{n}$ where $n$ is among $11$ possible numbers, and
\cite{PYL19} showed that if $\val(\Ga)=7$ and the vertex stabilizer is solvable
then $\Ga$ is an $\A_{n+1}$-arc-transitive Cayley graph
on $\A_{n}$ with $n\in\{7,21,63,84\}$.
Very recently, Yin, Feng, Zhou and Chen \cite{YFZC21} proved that
if $\val(\Ga)$ is a prime greater than $7$ and the vertex stabilizer is solvable, then $\Ga$
is either an $\A_{n+1}$-arc-transitive Cayley graph on $\A_{n}$ or one of the three exceptions.

From the above mentioned results in \cite{DF19,DFZ17,PYL19,XFWX05,XFWX07,YFZC21} on $s$-arc-transitive nonnormal Cayley graphs on nonabelian simple groups of certain valencies one can observe an interesting phenomenon: most of these graphs turn out to be Cayley graphs on alternating groups.
This motivates a natural problem as follows.


\vskip0.1in
{\noindent{\bf{Problem A.}}
Study $s$-arc-transitive nonnormal Cayley graphs on alternating groups for arbitrary valency.}

\vskip0.1in
In this paper, we give a characterization of the graphs in Problem A with $s=2$. For a group $G$, the socle $\Soc(G)$ of $G$ is the product of all minimal normal subgroups of $G$. For a $G$-vertex transitive graph $\Ga$ and a normal subgroup $M$ of $G$, the quotient graph $\Ga_M$  associated with $M$ is a graph with the set of orbits of $M$ as vertex set, and two orbits $O_1,O_2$ are adjacent if there is an edge in $\Ga$ with two endpoints in $O_1,O_2$ respectively. The graph $\Ga$ is called a \emph{normal cover} of $\Ga_M$ if the valency $\val(\Ga)=\val(\Ga_M)$.

\begin{theorem}\label{th:1}
Let $\Ga$ be a connected $(G,2)$-arc-transitive Cayley graph on an alternating group $H\cong\A_n$ with $n\ge 5$ and $H\leq G$, and let $\a$ be a vertex of $\Ga$. Suppose that $H$ is not normal in $G$. Then   one of the following holds:
\begin{enumerate}[\rm (a)]
\item $\Soc(G)=\A_{n+1}$, $\Soc(G)_\a$ is a regular subgroup of $\A_{n+1}$.
\item $G$ has a normal subgroup $N \cong \A_{n+2}$ such that $H<N$, $\Ga$ is  $(N,2)$-arc-transitive, $N_\a$ is a sharply $2$-transitive subgroup of $\A_{n+2}$, and $N_\a$ acts faithfully on $\Ga(\a)$; in particular, $\val(\Ga)=n+2$ is a prime power.
\item $\Ga=\K_{m}$ with $m=n!/2$, and either $(\Soc(G),\Soc(G)_\a)=(\A_{  m},\A_{m-1})$ or $(G,G_\a,n)=(\PGL_2(59),\AGL_1(11),5)$.
\item $G$ has a nontrivial maximal intransitive normal subgroup $M$ such that $G/M$ is almost simple, $\Ga_M$ is $(G/M,2)$-arc-transitive, and $\Ga$ is a normal cover of $\Ga_M$; moreover, one of the following is true, where $v$ is a vertex of $\Ga_M$:
\begin{enumerate}[\rm (d.1)]
\item $\Soc(G/M)=\A_{n}$ such that $n \geq 16$ is a power of $2$, and $\Soc(G/M)_v$ contains a Sylow $2$-subgroup of $\Soc(G/M)$.
\item $\Soc(G/M)=\A_{n+1}$, and $\Soc(G/M)_v$ is a transitive subgroup of  $\A_{n+1}$;
\item $\Ga_M=\K_m$, $\Soc(G/M)=\A_m$ such that $m$ is the index of a subgroup of $\A_n$, and $\Soc(G/M)_v=\A_{m-1}$.
\end{enumerate}
\end{enumerate}
\end{theorem}

A group $G$ is said to be \emph{almost simple} if $\Soc(G)$ is a nonabelian simple group. For Problem A, if we impose that the $2$-arc-transitive automorphism group of the graph is almost simple, then we have the following corollary of Theorem~\ref{th:1}.

\begin{corollary}\label{Coro-1}
Let $\Ga$ be a connected $(G,2)$-arc-transitive Cayley graph on an alternating group $H\cong\A_n$ with $n\ge 5$ and $H\leq G$, and let $\a$ be a vertex of $\Ga$. Suppose that $G$ is almost simple with socle $T$. Then one of the following holds:
\begin{enumerate}[\rm (a)]
\item $T=\A_{n+1}$ and $T_\a$ is a regular subgroup of $ \A_{n+1}$;
\item $T=\A_{n+2}$ such that $\Ga$ is $(T,2)$-arc-transitive, $T_\a$ is a sharply $2$-transitive subgroup of $\A_{n+2}$ and acts faithfully on $\Ga(\a)$, and $\val(\Ga)=n+2$ is a prime power;
\item $\Ga=\K_{n!/2}$.
\end{enumerate}
 \end{corollary}

As introduced above, there are a number of $(\A_{n+1},2)$-arc-transitive Cayley graphs on $\A_n$ known in the literature, which give rise to examples of case~(a) in Corollary~\ref{Coro-1}. However, to the best of our knowledge, there is no example of $(\A_{n+2},2)$-arc-transitive Cayley graph on $\A_n$ in the literature. Note that the vertex stabilizer $T_\a$ in case~(b) of Corollary~\ref{Coro-1}, as a sharply $2$-transitive group, is an affine group with degree a prime power $p^d$, and is either a subgroup of $\AGaL_1(p^d)$ or with $p^d \in \{ 5^2,7^2,11^2,19^2,23^2,29^2,59^2\}$ (see for example \cite[Chapter XII \S9]{Huppert82b}).
In this paper, we construct the first infinite family of $(\A_{n+2},2)$-arc-transitive Cayley graphs on $\A_n$ with vertex stabilizer a subgroup of $\AGaL_1(p^d)$. 

\begin{construction}\label{Exa2}
Let $q$ be a prime power such that $q\equiv3\pmod{4}$, and let $\tau\colon v\mapsto v^q$ be the filed automorphism of $\bbF_{q^2}$ of order $2$. Take $G=\Sym(\bbF_{q^2})$, $K=\AGL_1(q^2){:}\l\tau\r\le\AGaL_1(q^2)<G$, and $g\in G$ such that $0^g=0$ and $v^g=v^{-1}$ for all $v\in\bbF_{q^2}^\times$. Then let $\Ga=\Cos(G,K,g)$ be the coset graph (see Subsection~\ref{subsec1} for definition) given by the triple $(G,K,g)$.
\end{construction}

\begin{theorem}\label{Thm2}
Let $q$, $G$ and $\Ga$ be as in Construction~$\ref{Exa2}$. Then $\Ga$ is a connected $(\A_{q^2},2)$-arc-transitive Cayley graph on $\A_{q^2-2}$ with valency $q^2$.
\end{theorem}

The rest of the paper is organized as follows. 
After a preliminary section, we deal with the vertex-quasiprimitive case of Theorems~\ref{th:1} in Section~\ref{sec1}, which leads to a proof of Theorems~\ref{th:1} in Section~\ref{sec2}.
Then in Section~\ref{sec3} we  prove Theorem~\ref{Thm2}.

\section{Preliminaries}

The notations used in this paper are standard.
As in \cite{Atlas},
we sometimes use $n$ to denote a cyclic group of order $n$,
use $[n]$ to denote a group of order $n$, and use $p^n$ with $p$ a prime to denote the elementary abelian group of order $p^n$.
For a positive integer $n$ and prime number $p$, let denote by $n_p$ the largest $p$-power dividing $n$. 
For a graph $\Ga$, let $V(\Ga)$ denote the vertex set of $\Ga$, and let $\Ga(v)$ denote the set of neighborhoods of $v \in V(\Ga)$.

\subsection{$2$-arc-transitive graphs}

The following result shows that the $s$-arc-transitivity of graph with $s\geq2$ is inherited by normal quotients, see
Praeger~\cite[Theorem 4.1]{Praeger92}.

\begin{theorem}\label{th:Praeger}
Let $\Ga$ be a connected $(G,s)$-arc-transitive graph with $s\geq  2$. Suppose that $G$ has a  normal subgroup  $M$ with at least three orbits on $V(\Ga)$.
Then $\Ga_M$ is $(G/M,s)$-arc-transitive with $(G/M)_v \cong G_\a$, where $v \in V(\Ga_M)$ and $\a \in V(\Ga)$.
Moreover, $M$ is semiregular on $V(\Ga)$, and $\Ga$ is a normal cover of $\Ga_M$.
\end{theorem}

Let $\Ga$ be a graph and let $G\le\Aut(\Ga)$. For adjacent vertices $\a$ and $\b$ of $\Ga$,
let $G_{\a}^{[1]}$ be the kernel of $G_{\a}$ acting on $\Ga(\a)$,
and let $G_{\a\b}^{[1]}=G_{\a}^{[1]}\cap G_{\b}^{[1]}$.
The following theorem is a well-known result of Weiss~\cite{Weiss91} on
local action of 2-arc-transitive graphs.

\begin{theorem}\label{Weiss}
Let $\Ga$ be a connected $(G,2)$-transitive graph with $G\le\Aut(\Ga)$, and let $\a$ and $\b$ be adjacent vertices of $\Ga$.
Then one of the following holds:
\begin{enumerate}[{\rm (a)}]
\item $G_{\a\b}^{[1]}=1$, and $G_\a^{[1]}\cong(G_{\a}^{[1]})^{\Ga(\b)}\trianglelefteq G_{\a\b}^{\Ga(\b)}\cong G_{\a\b}^{\Ga(\a)}$;
\item $G_{\a\b}^{[1]}$ is a nontrivial $p$-group for some prime $p$, and there exist integers $d\ge2$ and $f\ge1$
such that $G_{\a}^{\Ga(\a)}\trianglerighteq\PSL_d(p^f)$ and $\val(\Ga)=(p^{df}-1)/(p^f-1)$.
\end{enumerate}
\end{theorem}

The following lemma is also well known. We include a proof here as it is not lengthy.

\begin{lemma}\label{Insoluble}
Let $\Ga$ be a connected graph, let $G$ be a vertex-transitive subgroup of $\Aut(\Ga)$, and let $\a\in V(\Ga)$.
Then every composition factor of $G_\a$ is isomorphic to a section of $G_\a^{\Ga(\a)}$.
\end{lemma}

\begin{proof}
Let $d$ be the diameter of $\Ga$.
For $i\in\{0,1,\dots,d\}$, let $V_i$ be the set of vertices of distance $i$ from $\a$ in $\Ga$,
let $U_i=V_0\cup V_1\cup\dots\cup V_i$, and
let $K_i$ be the kernel of $G_\a$ acting on $U_i$.
By the connectivity of $\Ga$ we have $U_d=V(\Ga)$, and thus
\[
1=K_d\trianglelefteq K_{d-1}\trianglelefteq\dots\trianglelefteq K_1\trianglelefteq K_0=G_\a.
\]
Moreover, for $i\in\{1,\dots,d\}$,
\[
K_{i-1}/K_i\cong K_{i-1}^{V_i}.
\]
Let $\Delta$ be an orbit of $K_{i-1}$ on $V_i$, and let $\gamma\in\Delta$.
Then there exists $\beta\in U_{i-1}$ such that $\gamma\in\Ga(\beta)$.
It follows that $\gamma^k\in\Ga(\beta^k)=\Ga(\beta)$ for each $k\in K_{i-1}$.
Hence $\Delta=\gamma^{K_{i-1}}\subseteq\Ga(\beta)$, and so $K_{i-1}^\Delta$ is a homomorphic image of $K_{i-1}^{\Ga(\beta)}$.
Note that $K_{i-1}^{\Ga(\beta)}\le G_\beta^{\Ga(\beta)}\cong G_\a^{\Ga(\a)}$ since $G$ is vertex-transitive.
We then conclude that $K_{i-1}^\Delta$ is (isomorphic to) a section of $G_\a^{\Ga(\a)}$.
Since $\Delta$ is an arbitrary orbit of $K_{i-1}$ on $V_i$,
this implies that $K_{i-1}/K_i\cong K_{i-1}^{V_i}$ is a section of $G_\a^{\Ga(\a)}$.
Thus every composition factor of $G_\a$ is a section of $G_\a^{\Ga(\a)}$.
\end{proof}

Recall that the $2$-transitive groups are known (see for instance~\cite[Chapter~7]{Cameron}).
In particular, a $2$-transitive group is either affine or almost simple.

\begin{lemma}\label{(L,2)-Trans}
Let $\Ga$ be a connected $(G,2)$-arc-transitive graph, let $N$ be a vertex-transitive normal subgroup of $G$, and let $\a\in V(\Ga)$.
Then the following hold:
\begin{enumerate}[{\rm (a)}]
\item if $N_{\a}^{\Ga(\a)}$ is almost simple with socle not isomorphic to $\PSL_2(8)$, then $\Ga$ is $(N,2)$-arc-transitive;
\item if $N_{\a}$ is almost simple with socle not isomorphic to $\PSL_2(8)$, then $\Ga$ is $(N,2)$-arc-transitive.
\end{enumerate}
\end{lemma}

\begin{proof}
Since $\Ga$ is $(G,2)$-arc-transitive, $G_\a^{\Ga(\a)}$ is $2$-transitive.
Suppose that $N_\a^{\Ga(\a)}$ is almost simple with socle not isomorphic to $\PSL_2(8)$.
It follows that $N_\a^{\Ga(\a)}$ is a normal subgroup of the $2$-transitive group $G_\a^{\Ga(\a)}$.
Then by the classification of $2$-transitive groups, $G_\a^{\Ga(\a)}$ is almost simple with $\Soc(G_\a^{\Ga(\a)})=\Soc(N_\a^{\Ga(\a)})$, and $N_\a^{\Ga(\a)}$ is $2$-transitive.
Since $N$ is vertex-transitive, we then conclude that $\Ga$ is $(N,2)$-arc-transitive.
This proves part~(a).

Now suppose that $N_{\a}$ is almost simple with socle not isomorphic to $\PSL_2(8)$.
Then $N_\a^{\Ga(\a)}$ is nonsolvable by Lemma~\ref{Insoluble}.
Since $N_\a^{\Ga(\a)}$ is a homomorphic image of the almost simple group $N_\a$, this implies that $N_\a^{\Ga(\a)}\cong N_\a$.
Thus the conclusion of part~(a) shows that $\Ga$ is $(N,2)$-arc-transitive, proving part~(b).
\end{proof}

\subsection{Coset graphs and orbital graphs}\label{subsec1}

Coset graph and orbital graph are two useful tools to construct and understand general arc-transitive graphs.

Let $G$ be a group, $K$ a core-free subgroup of $G$ (namely, $K$ contains no nontrivial normal subgroup of $G$) and $g\in G\setminus K$. Define the \emph{coset graph} $\Cos(G,K,g)$ to be a graph with vertex set $[G:K]$ (the set of right cosets of $K$ in $G$) such that $Kx$ is adjacent to $Ky$ with $x,y\in G$ if and only if $yx^{-1}\in KgK$. The following assertion is due to Sabidussi \cite{Sab64}.


\begin{lemma}\label{Sab64}
Let $\Ga$ be a connected $G$-arc-transitive graph of valency $d$, and let $\a\in V(\Ga)$.
Then $\Ga\cong\Cos(G,K,g)$ for $K=G_\a$ and some $2$-element $g$ satisfying:
\begin{equation}\label{Eqn4}
g\in\Nor_G(K\cap K^g),\ \ g^2\in K,\ \ \l K,g\r=G,\ \ |K|/|K\cap K^g|=d.
\end{equation}
In particular, if $\b\in\Ga(\a)$ then $\l G_\a,\Nor_G(G_{\a\b})\r=G$.

Conversely, if $H$ is a core-free subgroup of a group $X$ and $x$ is an element of $X$
such that the triple $(X,H,x)$ $($as $(G,K,g)$ there$)$ satisfies~\eqref{Eqn4},
then $\Cos(X,H,x)$ is a connected $X$-arc-transitive graph of valency $d$.
\end{lemma}

Let $G$ be a transitive permutation group on a set $\Ome$. Then $G$ naturally acts on $\Ome \times \Ome$ by
\[ 
(\a,\b)^g=(\a^g,\b^g)\ \ \text{for}\ g \in G\ \text{and} \   \a,\b \in \Ome.
\]
The orbits of $G$ on $\Ome \times \Ome$ are called orbitals of $G$. 
By the transitivity of $G$, each orbital of $G$ corresponds to a orbit of $G_{\a}$ on $\Ome$, which is called a \emph{suborbit} of $G$. 
Let  $\Delta$ be a suborbit of $G$, and let $\beta \in \Delta$. 
We say that $\Delta$ is \emph{nontrivial} if $\Delta \neq \{ \a\}$, and $\Delta$ is \emph{self-paired} if $(\a,\b)^G=(\b,\a)^G$. If $\Delta$ is nontrivial, then the associated \emph{orbital graph} $\Ga(\Delta)$ of $\Delta$ is the graph with vertex set $\Ome$ and edge set $\{\a,\b\}^G:=\{ \{ \a^g,\b^g\}: g \in G\}$.
Clearly, if $\Delta$ is nontrivial and self-paired, then $\Ga(\Delta)$ is $G$-arc-transitive.
Conversely, every $G$-arc-transitive graph can arise in this way.
Note that if $\Ga(\Delta)$ is connected, then  $\langle G_\a,g\rangle=G$ for any $g \in G$ such that $\alpha^g=\beta$. 

The next lemma will be needed in Section~\ref{sec1}.

\begin{lemma}\label{lm-AnAn-2Sn-2}
Let $G=\A_n$ be a transitive permutation group on a set $\Ome$, where $n\ge 5$, and let $\a\in \Ome$. 
Suppose that $G_\a\cong\A_{n-2}$ or $\Sy_{n-2}$. 
Then there is no connected $(G,2)$-arc-transitive associated orbital graph for any non-trivial suborbit of $G$ on $\Omega$.
\end{lemma}

\begin{proof} 
Suppose for a contradiction that there exists a connected $(G,2)$-arc-transitive associated orbital graph $\Ga$ with a suborbit $\Delta$ of $G$ on $\Ome$. 
Then $\Delta$ is nontrivial and self-paired. 

First assume that $G_\a\cong\A_{n-2}$. 
Then the action of $G$ on $\Ome$ can be identified with
the natural action of $G$ on $\Phi^{(2)}$,
the set of ordered pairs in the set $\Phi:=\{1,2,\dots,n\}$.
Identify $\Ome$ with $\Phi^{(2)}$, and assume without loss of generality that $\a=(1,2)$. 
Let $\b=(i,j)\in\Delta$.
Then $\Delta=\b^{G_\a}$, and $\b\neq\a$ as $\Delta$ is nontrivial.
If $\{i,j\}=\{1,2\}$, then $\b=(2,1)$, and so the suborbit $\Delta=\b^{G_\a}=\{\beta\}$ has length $1$, contradicting the connectivity of $\Ga$. 
If $\{i,j\}\cap=\{1,2\}=\emptyset$, then $G_{\a\b}\cong\A_{n-4}$ is not maximal in $\G_\a$, and so the action of $G_\a$ on $[G_\a:G_{\a\b}]$ is not $2$-transitive, contradicting the $(G,2)$-arc-transitivity of $\Ga$.
Hence $|\{i,j\}\cap\{1,2\}|=1$.
Assume without loss of generality that $\beta=(1,3)$.
Let $g=(2,3)(4,5)$. Then $g$ maps $\a$ to $\b$, and so $\langle G_\a,g \rangle=G$ by the connectivity of $\Ga$.
However, $\langle G_\a,g \rangle$ fixes $1$ as both $G_\a$ and $g$ fixes $1$, a contradiction. 

Next assume that $G_\a\cong\Sy_{n-2}$. Then the action of $G$ on $\Ome$ can be identified with
the natural action of $G$ on $\Phi^{\{2\}}$,
the set of $2$-subsets of the set $\Phi:=\{1,2,\dots,n\}$.
Identify $\Ome$ with $\Phi^{\{2\}}$, and assume without loss of generality that $\a=\{1,2\}$. 
Let $\b=\{i,j\}\in\Delta$.
Then $\Delta=\b^{G_\a}$, and $\b\neq\a$ as $\Delta$ is nontrivial.
If $\{1,2\}\cap\{i,j\}=\emptyset$,
then $G_{\a\b}\cong\Sy_{n-4}.\ZZ_2$.
If $|\{i,2\}\cap\{i,j\}|=1$,
then $G_{\a\b}\cong\A_{n-3}$.
In either case, $G_{\a\b}$ is not maximal in $G_\a$, and so the action of $G_\a$ on $[G_\a:G_{\a\b}]$ is not $2$-transitive, contradicting the $(G,2)$-arc-transitivity of $\Ga$.
\end{proof}

\subsection{Quasiprimitive permutation groups with a transitive alternating group}

A transitive permutation group is said to be \emph{quasiprimitive} if all of its nontrivial normal subgroups are transitive. 
The following classification of quasiprimitive permutation groups containing a transitive alternating group was obtained by the second-named author \cite{Xia17}.

\begin{proposition}\label{Xia}
Let $G$ be a quasiprimitive permutation group on $\Omega$, and let $\a\in\Ome$.
If $G$ contains a transitive subgroup $H=\A_n$ with $n\ge5$, then one of the following holds.
\begin{enumerate}[{\rm (a)}]
\item $G$ is almost simple with socle $L$ such that either $L=H$, or $L=HL_\a$ satisfies one of the following:
\begin{enumerate}
\item[{\rm (a.1)}] $L=\A_{n+k}$ with $1\le k \le5$, and $L_\a$ is $k$-transitive on $n+k$ points;
\item[{\rm (a.2)}] $L=\A_{m}$ and $L_\a=\A_{m-1}$, where $m$ is the index of a subgroup of $A_n$;
\item[{\rm (a.3)}] $(L, n, L_\a)$ lies in Table~$\ref{Sporadic}$.
\end{enumerate}
\item $G$ is primitive with socle $\A_n\times\A_n$, and $H$ is regular.
\item $n=6$, $G$ is primitive with socle $\A_6\times\A_6$, and $G\le\Sy_6\wr\Sy_2$ by the product action on $6^2$ points.
\end{enumerate}
\end{proposition}

\begin{table}[ht]
\[\begin{array}{llll} \hline
\text{Row} & L &  n  &  L_\a  \\ \hline
1& \A_6 &  5  &  \A_4,\,\Sy_4  \\
2& \A_{10} &  6  &  \A_8,\,\Sy_8  \\
3& \A_{15} &  7,\,8  &  \A_{13},\,\Sy_{13}  \\
4& \M_{12} &  5  &  \M_{11}  \\
5& \PSL_2(11)&  5  &  11,\,11{:}5  \\
6& \PSL_2(19)&  5  &  19{:}9  \\
7& \PSL_2(29)&  5  &  29{:}7,\,29{:}14  \\
8& \PSL_2(59)&  5  &  59{:}29  \\
9& \PSL_4(3)&  6  &  3^3{:}\PSL_3(3)  \\
10& \PSU_3(5)&  7  &  5^{1+2}_{+}{:}8  \\
11& \PSp_4(3)&  6 &  3^{1+2}_{+}{:}\Q_8,\,3^{1+2}_{+}{:}2.\A_4 \\
12& \Sp_6(2)&  6,\,7,\,8 &  \PSU_3(3){:}2 \\
13& \Sp_6(2)&  8 &  3^{1+2}_{+}{:}8{:}2,\,3^{1+2}_{+}{:}2.\Sy_4,\,\PSL_2(8),\,\PSL_2(8){:}3,\,\PSU_4(2){:}2 \\
14& \Sp_8(2)&  6,\,7,\,8,\,9,\,10 & \mathrm{SO}^-_8(2) \\
15& \Omega_7(3)&  8,\,9 &  3^{3+3}{:}\PSL_3(3) \\
16& \Omega_7(3)&  9 &  3^3{:}\PSL_3(3),\,\PSL_4(3),\,\PSL_4(3){:}2,\,\G_2(3) \\
17& \Omega^{+}_8(2)&  6,\,7,\,8, \,9 & \Sp_6(2) \\
18& \Omega^{+}_8(2)& 8 & \A_9  \\
19& \Omega^{+}_8(2)& 8,\,9 & \PSU_4(2),\,\PSU_4(2){:}2,\,3\times\PSU_4(2),\,(3\times\PSU_4(2)){:}2 \\
20& \Omega^{+}_8(2)& 9 & 2^4{:}15.4,\,2^6{:}15,\,2^6{:}15.2,\,2^6{:}15.4,\,\A_8,\,\Sy_8\\
21& \Omega^{+}_8(2)& 9 & 2^4{:}\A_5\le L_\a\le2^6{:}\A_8\\
22& \POm^{+}_8(3)& 8,\,9 & 3^6{:}\PSL_4(3) \\
23& \POm^{+}_8(3)& 9 & \Omega_7(3)\\
24& \Omega^{-}_{10}(2)& 12 & 2^8{:}\Omega^{-}_8(2)\\
 \hline
\end{array}\]
\caption{Exceptional quasiprimitive groups with a transitive subgroup $\A_n$}\label{Sporadic}
\end{table}

Let us outline the proof of Theorem~\ref{th:1}, in which Proposition~\ref{Xia} will play a crucial role.
Let $\Ga$ be a connected $(G,2)$-arc-transitive Cayley graph on an alternating group $H\cong\A_n$ with $n\ge 5$
and $H\leq G$, and let $\a$ be a vertex of $\Ga$.
Take $M$ to be a maximal intransitive normal subgroup of $G$, and let $X=G/M$ and $Y=HM/H$. 
From Theorem \ref{th:Praeger} we see the quotient graph $\Ga_M$ is $X$-arc-transitive and $X$-vertex-quasiprimitive, and $X$ contains a vertex-transitive subgroup $Y\cong H$. 
Thus the candidates for the triple $(X,Y,X_v)$ are given in Proposition~\ref{Xia}. We investigate those candidates and give a characterization of such graphs $\Ga_M$ in Section\ref{sec1}. 
We shall see that most of the candidates for $(X,Y,X_v)$ satisfy $\Soc(X)=\A_{n+k}$ with $1\leq k \leq 5$, $Y=\A_n$ and that $X_v$ is a $k$-transitive group on $n+k$ points, as in Case~(a.1) of Proposition~\ref{Xia}. 
Moreover, the possibility for $k \geq 3$ will be excluded, and the graph $\Ga$ arising from this case must will be shown to satisfy part~(b) of Theorem \ref{th:1}. 
In Section 4, we consider the cover of $\Ga_M$ and complete the proof of Theorem~\ref{th:1}, where the same technique as in~\cite{DFZ17} will be used.

\subsection{Some technical lemmas}

Let $G$ be a permutation group on a set $\Ome$.
Recall that $G$ is said to be \emph{semiregular} on $\Ome$
if $G_{\a}=1$ for each $\a\in\Ome$.
A nontrivial cyclic subgroup $\l g\r$
of $G$ is semiregular on $\Ome$ if and only if
$g$ can be expressed as a disjoint product
of $s$ cycles of length $t$ such that $st=|\Ome|$ and $t>1$.

\begin{lemma}\label{Normalizer}
Let $G$ be a permutation group on a set $\Ome$,
let $H$ be a subgroup of $G$,
and let $H_\a$ be the stabilizer of a point $\a\in\Ome$ in $H$.
If $H_\a$ contains an element acting fixed-point-freely on $\Ome\setminus\{\a\}$,
then $\Nor_G(H_\a)\le G_\a$.
In particular, if $H$ is $2$-transitive then $\Nor_G(H_\a)\le G_\a$.
\end{lemma}

\begin{proof}
Let $g\in\Nor_G(H_\a)$, and let $h\in H_\a$ such that $h$ has no fixed-point on $\Ome\setminus\{\a\}$.
Then $\mathsf{Fix}(h)=\{\a\}$, and so $\mathsf{Fix}(h^g)=\{\a^g\}$.
Since $h^g\in (H_\a)^g=H_\a$, we have $\a\in\mathsf{Fix}(h^g)$.
It follows that $\a\in\{\a^g\}$.
Hence $\a^g=\a$, that is, $g\in G_\a$.
Thus $\Nor_G(H_\a)\le G_\a$, as required.

Now suppose that $H$ is $2$-transitive.
Then $H_\a$ is transitive on $\Ome\setminus\{\a\}$.
By a theorem of Jordan, there exits $h\in H_\a$ such that $h$ has no fixed-point on $\Ome\setminus\{\a\}$.
This implies $\Nor_G(H_\a)\le G_\a$ by the above conclusion.
\end{proof}

Recall that a \emph{section} of a group $G$ is a quotient of a subgroup of $G$.

\begin{lemma}\label{lm:miniM}
Let $B=M{:}H$ where $H=\A_{n}$ with $n\geq 9$. Suppose that $|M|_r<r^{n-2}$ for each prime divisor $r$ of $|M|$ and that $M$ has no section isomorphic to $H$. Then $B=M \times H$.
\end{lemma}

\begin{proof}
Suppose for a contradiction that $B \neq M \times H$. Let
\[
1=M_0 <M_1<\dots<M_s=M\lhd B
\]
be a normal series of $B$ such that $M_{i + 1}/ M_{ i}$ is minimal normal in $B/M_{i}$ for each $0 \leq i \leq s-1$.
Since $B \neq M \times H$, there exists $0\leq j \leq s-1$ such that $M_{j+1}H \neq  M_{j+1}\times H$ but $M_iH = M_i\times H$ for each $0 \leq i \leq j$. Since $M\cap H=1$, we have $M_{j}H/M_j\cong H \cong \A_n$ and $M_{j+1}H/M_j=(M_{j+1}/M_j){:}(M_{j}H/M_j)$.  
If the conjugation action of $M_{j}H/M_j$ on $M_{j+1}/M_j$ is trivial, then $M_{j+1}H/M_j=(M_{j+1}/M_j) \times (M_{j}H/M_j)$ and so $M_jH \lhd M_{j+1}H$.
If this is the case, then since $M_jH=M_j\times H$ and $M$ has no section isomorphic to $\A_{n}$, it follows that $H$ is characteristic in $M_jH $ and thus normal in $M_{j+1}H$, which implies that $M_{j+1}H=M_{j+1} \times H$, a contradiction. 
Hence the conjugation action of $M_{j}H/M_j$ on $M_{j+1}/M_j$ is nontrivial, and is thus faithful by the simplicity of $M_jH/M_j$.
Therefore, $\A_n \cong M_{j}H/M_j \leq \Aut(M_{j+1}/M_j)$.

Write $M_{j+1}/M_j=T_1\times \cdots \times T_m $ with $T_1 \cong \cdots \cong T_m \cong T$ for some simple group $T$.  If $T \cong \ZZ_p$ for some prime $p$, then $\A_n\le\Aut(\ZZ_p^m)\cong\GL_m(p)$, and by~\cite[Proposition~5.3.7]{K-Lie} we have $m\ge n-2$, which yields that $|M|_p \geq|M_{j+1}/M_j|=p^m\geq p^{n-2}$, contradicting the condition of the lemma. Hence $T$ is a nonabelian simple group. Then since $\A_n \leq \Aut(T^m)\cong \Aut(T) \wr \Sy_m$ and $M$ has no section isomorphic to $\A_n$, we derive that $\A_n \cap  \Aut(T) ^m=1$ and so $\A_n\lesssim\Sy_m$. This implies $n\leq m$. Then for any prime divisor $r$ of $|T|$ we have $|M|_r\geq|T|^m\geq r^m\geq r^n$, again contradicting the condition of the lemma.
 \end{proof}

Let $G$ be a $2$-transitive group of degree $n$ not containing $\A_n$.
Pyber \cite{Pyber93} gave the bound $|G|\leq n^{c\log^2n}$ for some constant $c >0$, and in particular, $|G|\leq n^{8\lceil 4\log n \rceil\log n}$ if $n\geq 400$, where $\log n$ means logarithm to the base $2$. 
His proof does not rely on the classification of finite simple groups (CFSG for short), and he noted that one can prove $|G| \lessapprox n^{(1+o(1))\log n}$ by using CFSG. In the following lemma we shall prove that $|G|<n(n-1)2^{n-4}$. Computation shows that
\[ 
n(n-1)2^{n-4} < n^{8\lceil 4\log n \rceil\log n}\text{ for } n < 168840. 
\]
Thus if we use Pyber's result, we would still need to investigate $2$-transitive groups of degree less than $168840$. 
Hence for the convenience of the proof we shall use CFSG. 
Then the list of $2$-transitive groups can be found in \cite[Tables~7.3~and~7.4]{Cameron}.
Note that a $2$-transitive group is affine or almost simple.

\begin{lemma}\label{lm:order2trans}
Let $G$ be a $2$-transitive permutation group with degree $n\geq11$ not containing $\A_n$. Then $ |G| <n(n-1)2^{n-4}$.
\end{lemma}

\begin{proof}
Computation shows that $n^{1.5\log n}<n(n-1)2^{n-4}$ for $n\geq 31$.
Hence for $n\geq 31$ it suffices to show $|G|\leq n^{1.5\log n}$.
By the classification of $2$-transitive groups, it is straightforward to verify the conclusion for $n\leq30$. 
Thus we assume $n\geq 31$ for the rest of the proof.

\underline{Case~1}: $G$ is affine. 
Then $G \leq \AGL_d(p)$ for some prime $p$ and integer $d$ such that $n=p^d$, and so $d+1 =\log_p n+1 \leq \log n+1 <1.5 \log n$.
Consequently,
\[ 
|G| \leq |\AGL_d(p)|=p^d\prod_{i=0}^{d-1} (p^d-p^i)<p^d\prod_{i=0}^{d-1}p^d=p^{d(d+1)}=n^{d+1}<n^{1.5\log n}.  
\]

\underline{Case~2}: $\Soc(G)=\PSL_d(q)$ with $n=(q^d-1)/(q-1)$, where $d\geq 2$ and $q=p^e$ for some prime $p$ and integer $e$. 
If $d=2$, then
\[  
|G|\leq e(d,q-1)|\PSL(2,q)|=q(q+1)(q-1)e\leq q(q+1)2^{q-3}=n(n-1)2^{n-4}.  
\]
Now let $d\geq 3$.
Then $|G| \leq 2e(d,q-1)|\PSL_d(q)|$. The candidates of the pair $(d,q)$ such that $31\leq n \leq 63$ are $(5, 2)$, $(6, 2)$, $(4, 3)$, $(3, 5)$, $(3, 7)$, and direct calculation shows that the lemma is true for these candidates. For $n\geq 64$, since $n>q^{d-1}$, we have $d+2 < \log_q n+3 \leq \log n+3 \leq 1.5\log n$ and $q^{(d-1)(d+2)}<n^{d+2}<n^{1.5\log n}$, which leads to
\[  
|G|\leq 2eq^{d-1}\prod_{i=0}^{d-2}(q^d-q^i) <2eq^{d-1}\prod_{i=0}^{d-2}q^d= 2eq^{d^2-1}\leq q^{d^2}<q^{(d-1)(d+2)}<n^{1.5\log n}.  
\]

\underline{Case~3}: $G=\Sp_{2d}(2)$ with $n=2^{2d-1}+2^{d-1}$ or $2^{2d-1}-2^{d-1}$, where $d\geq 3$. 
If $d=3$, then $G=\Sp_6(2)$, and so $|G|<31^{1.5\log31}\leq n^{1.5\log n}$ as $n\geq31$. 
If $d\geq4$, then $n\geq 2^{2d-1}-2^{d-1}>2^{2d-2}$ and $1.5 \log n>3d-3>d+2$, which implies that 
\[ 
|G|= 2^{d^2}\prod_{i=1}^{d}(2^{2i}-1) <2^{d^2}\prod_{i=1}^{d}2^{2i}=2^{d(2d+1)}\leq2^{(2d-2)\cdot (d+2)}<n^{d+2}<n^{1.5\log n}. 
\]

\underline{Case~4}: $\Soc(G)=\PSU_3(q)$ with $n=q^3+1$, where $q=p^e\geq4$ for some prime $p$ and integer $e$. 
In this case, we have
\[ 
|G| \leq 2(3,q+1)e|\PSU_3(q)| =2eq^3(q^3+1)(q^2-1)<q^9<(q^3+1)^3=n^3<n^{1.5\log n}. 
\]

\underline{Case~5}: $\Soc(G)=\Sz(q)$ with $n=q^2+1$, where $q=2^{2e+1} \geq 8$ for some integer $e$. 
Then
\[ 
|G| \leq |\Sz(q)|(2e+1)=(2e+1)q^2(q^2+1)(q-1)<q^6<(q^2+1)^3=n^3<n^{1.5\log n}. 
\]

\underline{Case~6}: $\Soc(G)=\Ree(q)$ and $n=q^3+1$, where $q=3^{2e+1} \geq 27$ for some integer $e$. 
In this case,
\[ 
|G| \leq |\Ree(q)|(2e+1)=(2e+1)q^3(q^3+1)(q-1)<q^8<(q^3+1)^3=n^3<n^{1.5\log n}.
\] 

For $G$ not in any of Cases~1--6, by the classification of $2$-transitive groups, there are finitely many candidates for $G$. For these candidates, one can directly verify the conclusion of the lemma. 
\end{proof}

\section{Vertex-quasiprimitive case}\label{sec1}

In this section we prove the following proposition.

\begin{proposition}\label{Quasi-case}
Let $\Ga$ be a connected $(G,2)$-arc-transitive graph, let $\a$ be a vertex of $\Ga$.
Suppose that $G$ is quasiprimitive on $V(\Ga)$ and has a vertex-transitive subgroup $H=\A_n$ with $n\ge5$.
Then $G$ is almost simple, and one of the following holds:
\begin{enumerate}[{\rm (a)}]
\item $\Soc(G)=H$;
\item $\Soc(G)=\A_{n+1}$, and $\Soc(G)_\a$ is a transitive subgroup of $\A_{n+1}$;
\item $\Soc(G)=\A_{n+2}$, $\Soc(G)_\a$ is a $2$-transitive subgroup of $\A_{n+2}$, the action of $G_\a$ on $\Ga(\a)$ is faithful, and $\Ga$ is $(\Soc(G),2)$-arc-transitive;
\item $\Ga=\K_m$, $\Soc(G)=\A_m$ such that $m$ is the index of a subgroup of $\A_n$, and $\Soc(G)_\a=\A_{m-1}$;
\item $\Ga=\K_{p+1}$, $H=\A_5$, $G=\PGL_2(p)$ with $p\in\{11,19,29,59\}$, and $G_\a=\AGL_1(p)$;
\item $\Ga=\K_{12}$, $H=\A_5$, $G=\M_{12}$, and $G_\a=\M_{11}$.
\end{enumerate}
\end{proposition}

Let $\Ga$, $G$, $\a$ and $H$ be as in the assumption of Proposition~\ref{Quasi-case}, and let $\b\in\Ga(\a)$. Then $G$ satisfies Proposition~\ref{Xia}.

\begin{lemma}\label{part23}
The pair $(G,H)$ does not satisfy part \emph{(b)} or \emph{(c)} of Proposition~$\ref{Xia}$.
\end{lemma}

\begin{proof}
Recall that the O'Nan-Scott-Praeger Theorem divides quasiprimitive permutation groups into eight types, see~\cite[Section~5]{Praeger1997}.
Since $\Ga$ is $(G,2)$-arc-transitive and $G$ is quasiprimitive on $V(\Ga)$, it is shown in~\cite{Praeger92}
(see also~\cite[Theorem~6.1]{Praeger1997}) that $G$ has type holomorph affine, almost simple, twisted wreath product or product action.

If $(G,H)$ satisfies part (b) of Proposition~\ref{Xia},
then $G$ is of type holomorph simple or simple diagonal, a contradiction.

Suppose that $(G,H)$ satisfies part (c) of Proposition~\ref{Xia}.
Then $G\le\Sy_6\wr\Sy_2$ acts primitively on $V(\Ga)$ by product action with $\Soc(G)=\A_6\times\A_6$ and $|V(\Ga)|=36$.
Hence
\[
(\A_5\times\A_5).2\le G_\a\le\Sy_5\wr\Sy_2,
\]
and $G_\a$ has a unique minimal normal subgroup $\A_5\times\A_5$.
Since $G_\a^{\Ga(\a)}$ is a homomorphic image of $G_\a$, it follows that either $G_\a^{\Ga(\a)}\cong G_\a$ or $G_\a^{\Ga(\a)}$ is solvable.
However, Lemma~\ref{Insoluble} implies that $G_\a^{\Ga(\a)}$ is nonsolvable.
Thus $G_\a^{\Ga(\a)}\cong G_\a$, from which we deduce that $G_\a^{\Ga(\a)}$ is not $2$-transitive, contradicting the $(G,2)$-arc-transitivity of $\Ga$.
\end{proof}

Now assume that $(G,H)$ satisfies part (a) of Proposition~\ref{Xia}, so that $G$ is almost simple.
Let $L=\Soc(G)$.
Then according to Proposition~\ref{Xia}, either $L=H$, or $L=HL_\a$ satisfies one of the following:
\begin{enumerate}
\item[{\rm (a.1)}] $L=\A_{n+k}$ with $1\le k \le5$, and $L_\a$ is $k$-transitive on $n+k$ points;
\item[{\rm (a.2)}] $L=\A_{m}$ and $L_\a=\A_{m-1}$, where $m$ is the index of a subgroup in $A_n$;
\item[{\rm (a.3)}] $(L, n, L_\a)$ lies in Table~\ref{Sporadic}.
\end{enumerate}

It is clear that case~(a.2) leads to part~(d) of Proposition~\ref{Quasi-case}. Thus we only need to deal with cases~(a.1) and~(a.3). These two cases will be treated in the following two subsections, respectively, after the next lemma.

\begin{lemma}\label{(L,2)-ArcT}
If $L_\a$ is nonsolvable, then $\Ga$ is $(L,2)$-arc-transitive.
\end{lemma}

\begin{proof}
Since $G$ is quasiprimitive on $V(\Ga)$, the normal subgroup $L$ is transitive on $V(\Ga)$.
Suppose for a contradiction that $\Ga$ is not $(L,2)$-arc-transitive.
Then $L_\a^{\Ga(\a)}$ is not $2$-transitive, and $L<G$.
As $\Ga$ is $(G,2)$-arc-transitive and $L_\a$ is nonsolvable, \cite[Corollary~1.2]{LSS15} asserts that $G_{\a}^{\Ga(\a)}$ is one of the groups:
\begin{equation}\label{Eqn2}
\PGaL_2(8),\ 19^2{:}(9\times\SL_2(5)),\ 29^2{:}(7\times\SL_2(5)),\ 29^2{:}(28\circ\SL_2(5)),\ 59^2{:}(29\times\SL_2(5)).
\end{equation}
Since $L$ is normal in $G$, we see that $L_{\a}^{\Ga(\a)}$ is normal in $G_{\a}^{\Ga(\a)}$.

Suppose that $L$ is not an alternating group. Then $(L, n, L_\a)$ lies in Table~\ref{Sporadic}. As $L_\a$ has a homomorphic image $L_{\a}^{\Ga(\a)}$ that is a normal subgroup of one of the groups in~\eqref{Eqn2}, inspecting the candidates for $L_\a$ in Table~\ref{Sporadic} shows that $L=\Sp_6(2)$ and $L_\a=\PSL_2(8)$ or $\PSL_2(8){:}3$. However, in this case we have $\Out(L)=1$, which implies that $G=L$, a contradiction.

Thus $L$ is an alternating group.
Since $L$ is transitive on $V(\Ga)$, we have $|G|/|G_\a|=|V(\Ga)|=|L|/|L_\a|$ and hence $|G_\a/L_\a|=|G/L|\in\{2,4\}$.
As a consequence, $|G_\a^{\Ga(\a)}/L_\a^{\Ga(\a)}|\in\{1,2,4\}$.
Moreover, since $L_\a^{\Ga(\a)}$ is not $2$-transitive, we have $L_\a^{\Ga(\a)}\ne G_\a^{\Ga(\a)}$.
Therefore, $|G_\a^{\Ga(\a)}/L_\a^{\Ga(\a)}|=2$ or $4$.
Now $G_{\a}^{\Ga(\a)}$ is one of the above five groups with a normal subgroup $L_\a^{\Ga(\a)}$ of index $2$ or $4$.
The only possibility is that $G_\a^{\Ga(\a)}=29^2{:}(28\circ\SL_2(5))$ and $L_\a^{\Ga(\a)}=29^2{:}(7\times\SL_2(5))$.
However, this implies that $L_\a^{\Ga(\a)}$ is $2$-transitive, a contradiction.
%
\end{proof}


\subsection{Case~(a.1)}


\begin{lemma}\label{AffineFaithful}
Let $L=\A_{n+k}$ with $2\le k\le5$. If $L_\a$ is an affine $k$-transitive subgroup of $L$, then the action of $G_\a$ on $\Ga(\a)$ is faithful.
\end{lemma}

\begin{proof}
Suppose for a contradiction that $L_\a$ is an affine $k$-transitive subgroup of $L$ while $G_\a$ is not faithful on $\Ga(\a)$. Then $G_\a$ is an affine $k$-transitive subgroup of $G$ with $\Soc(G_\a)=\C_p^m$ for some prime $p$ and integer $m$, and $G_\a^{[1]}\trianglerighteq\Soc(G_\a)$. Since $G_\a^{\Ga(\a)}\cong G_\a/G_\a^{[1]}$, it follows that $G_\a^{\Ga(\a)}$ is a quotient of $G_\a/\Soc(G_\a)$.

Suppose $G_{\a\b}^{[1]}=1$. Then by Theorem~\ref{Weiss}(a), the group $G_\a^{[1]}$ is isomorphic to a normal subgroup of $G_{\a\b}^{\Ga(\a)}$. In particular, $|G_\a^{[1]}|$ divides $|G_{\a\b}^{\Ga(\a)}|$ and thus divides $|G_\a^{\Ga(\a)}|$, which implies that $p^{2m}$ divides $|G_\a^{[1]}.G_\a^{\Ga(\a)}|=|G_\a|$. Checking the candidates (for example, in~\cite[Table~7.3]{Cameron}) of affine $2$-transitive groups $G_\a$ for the condition that $|\Soc(G_\a)|^2=p^{2m}$ divides $|G_\a|$, we obtain the following possibilities:
\begin{enumerate}[{\rm (i)}]
\item $G_\a/\Soc(G_\a)\trianglerighteq\SL_d(q)$ with $d\ge2$ and $q^d=p^m$;
\item $G_\a/\Soc(G_\a)\trianglerighteq\Sp_{2d}(q)$ with $d\ge2$ and $q^{2d}=p^m$;
\item $G_\a/\Soc(G_\a)\trianglerighteq\G_2(q)$ with $q^6=p^m$.
\end{enumerate}
Notice that $G_\a^{\Ga(\a)}$ is $2$-transitive and is a quotient group of $G_\a/\Soc(G_\a)$.
We see that (iii) is not possible because there is no 2-transitive permutation group with socle $\G_2(q)$.
Moreover, if (i) occurs, then $G_{\a}^{\Ga(\a)}$ is almost simple with socle $\PSL_d(q)$, and the largest normal $p$-subgroup of $G_{\a\b}^{\Ga(\a)}$ is of order $q^{d-1}$, contradicting the condition that $G_\a^{[1]}$ is isomorphic to a normal subgroup of $G_{\a\b}^{\Ga(\a)}$ with $G_\a^{[1]}\trianglerighteq\Soc(G_\a)=\C_p^m$. Now assume (ii) occurs. Then
the 2-transitivity of $G_{\a}^{\Ga(\a)}$ implies that it is almost simple with socle $\Sp_{2d}(q)$, where $q=2$, and $G_{\a\b}^{\Ga(\a)}$ is almost simple with socle $\POm_{2d}^+(2)$ or $\POm_{2d}^-(2)$. This also contradicts the condition that $G_\a^{[1]}$ is isomorphic to a normal subgroup of $G_{\a\b}^{\Ga(\a)}$ with $G_\a^{[1]}\trianglerighteq\Soc(G_\a)=\C_p^m$.

Therefore, $G_{\a\b}^{[1]}\neq1$. By Theorem~\ref{Weiss}, there exist prime $r$ and integers $e\ge2$ and $f\ge1$ such that $G_{\a\b}^{[1]}$ is a nontrivial $r$-group, $G_\a^{\Ga(\a)}\trianglerighteq\PSL_e(r^f)$ and $\val(\Ga)=(r^{ef}-1)/(r^f-1)$.
If $r\neq p$, then
\[
\C_p^m=\Soc(G_\a)\cong\Soc(G_\a)G_{\a\b}^{[1]}/G_{\a\b}^{[1]}\le G_\a^{[1]}/G_{\a\b}^{[1]}=(G_\a^{[1]})^{\Ga(\b)}\trianglelefteq(G_\b)^{\Ga(\b)}\cong(G_\a)^{\Ga(\a)},
\]
and so $p^{2m}$ divides $|G_\a^{[1]}.G_\a^{\Ga(\a)}|=|G_\a|$, which is not possible by the same argument as in the previous paragraph. Hence $r=p$, and so $\Soc(G_\a^{\Ga(\a)})=\PSL_e(p^f)$. Then since $G_\a^{\Ga(\a)}$ is a quotient of $G_\a/\Soc(G_\a)$, the only possibility for the affine $2$-transitive group $G_\a$ is
\begin{equation}\label{Eqn3}
\ASL_e(p^f)\trianglelefteq G_\a\le\AGaL_e(p^f)\ \text{ with }\ ef=m.
\end{equation}

Suppose $(e,p^f)\in\{(2,3),(2,4),(3,2),(3,3)\}$. Then viewing~\eqref{Eqn3} and $G=G_\a\A_n$ we have
\begin{align*}
(G,G_\a)\in\{&(\A_9,\ASL_2(3)),\,(\Sy_9,\AGL_2(3)),\,(\A_{16},\ASL_2(4)),\\
&(\A_{16},\AGL_2(4)),\,(\A_{16},\ASiL_2(4)),\,(\A_{16},\AGaL_2(4)),\\
&(\A_8,\ASL_3(2)),\,(\A_{27},\ASL_3(3)),\,(\Sy_{27},\AGL_3(3))\}.
\end{align*}
Moreover, $G_{\a\b}$ is a maximal subgroup of $G_\a$ such that the action of $G_\a$ on $[G_\a:G_{\a\b}]$ is $2$-transitive. However, for all the candidates computation in \magma~\cite{Magma} shows that $\l\Nor_G(G_{\a\b}),G_\a\r<G$, contradicting the connectivity of $\Ga$.

Thus $(e,p^f)\notin\{(2,3),(2,4),(3,2),(3,3)\}$. Let $V$ be an $e$-dimensional vector space over $\bbF_{p^f}$ such that $\Alt(V)\le G\le\Sym(V)$ and $\ASL(V)\trianglelefteq G_\a\le\AGaL(V)$, and let $N=\Soc(G_\a)$. Since $G_{\a\b}\ge G_{\a}^{[1]}\ge N$, we have $G_{\a\b}=N{:}(G_{\a\b})_0$. Moreover, $(G_{\a\b})_0$ is the stabilizer in $(G_\a)_0$ of a $1$-dimensional subspace or a $(e-1)$-dimensional subspace of $V$, as the action of $G_\a$ on $[G_\a:G_{\a\b}]$ is $2$-transitive. Hence $(G_{\a\b})_0=(P{:}Q).\mathcal{O}$ with $P=\C_p^{(e-1)f}$, $Q=\GL_{e-1}(p^f)$ and $\mathcal{O}=(G_\a)_0/\SL(V)\le\GaL(V)/\SL(V)$.

By Lemma~\ref{Sab64}, there exits $g\in\Nor_G(G_{\a\b})$ such that $\Ga\cong\Cos(G,G_\a,g)$ and $\l G_a,g\r=G$. Since $N\trianglelefteq G_{\a\b}$ and $g\in\Nor_G(G_{\a\b})$, we have $N^g\trianglelefteq G_{\a\b}$. Then $NN^g\leqslant NP$, as $NP$ is the largest normal $p$-subgroup of $G_{\a\b}$. If $N^g=N$, then $N$ is normal in $\l G_\a,g\r=G$, contradicting that $G$ is $\Alt(V)$ or $\Sym(V)$. Hence $N^g\neq N$, and so $NN^g/N$ is a nontrivial normal subgroup of $G_{\a\b}/N\cong(P{:}Q).\mathcal{O}$. Since $P$ is a minimal normal subgroup of $(P{:}Q).\mathcal{O}$, it follows that $NN^g/N\cong P$, which yields $NN^g=NP$. For $h\in\GL(V)$ and $w\in V$, define $t_{h,w}\in\AGL(V)$ by letting
\[
t_{h,w}\colon v\mapsto v^h+w\quad\text{for }\,v\in V,
\]
and identify $t_{h,0}$ with $h$. For each $h\in P$, since $t_{h,0}\in NP=NN^g$, there exists $x\in N$ such that $t_{h,0}x\in N^g$, and so there exists $w\in V$ such that $t_{h,w}=t_{h,0}t_{1,w}\in N^g$.

\underline{Case~1}: $(G_{\a\b})_0$ is the stabilizer in $(G_{\a})_0$ of a $(e-1)$-dimensional subspace of $V$. In this case, $PQ=\SL(V)_U$ for some $(e-1)$-dimensional subspace $U$ of $V$. Let $v_1,v_2,\dots,v_e$ be a basis of $V$ such that $U=\l v_2,\dots,v_e\r$. Since $|N|=|N^g|=(p^f)^e$ and $|NN^g|=|NP|=(p^f)^{2e-1}$, we have $|N\cap N^g|=p^f$. Moreover, since both $N$ and $N^g$ are normal in $G_{\a\b}$, their intersection $N\cap N^g$ is normalized by $PQ$. Then since the orbits of $PQ=\SL(V)_U$ on $V\setminus\{0\}$ are $U\setminus\{0\}$ and $V\setminus U$, we conclude that $N\cap N^g=\{t_{1,u}\mid u\in U\}$ with $e=2$. From $p^{ef}=p^m=n+k\ge n+2\ge7$ and $(e,p^f)\notin\{(2,3),(2,4)\}$ we deduce that $p^f\ge5$. Hence there exists $c\in\bbF_{p^f}^\times$ with $c\neq c^{-2}$. Let $h,\ell\in\SL(V)$ defined by
\[
v_1^h=v_1+v_2,\quad v_2^h=v_2,\quad v_1^\ell=cv_1,\quad v_2^\ell=c^{-1}v_2.
\]
Then $h\in P$ and $\ell\in Q$. As shown above, there exists $w\in V$ such that $t_{h,w}\in N^g$. Write $w=av_1+bv_2$ with $a,b\in\bbF_{p^f}$. Since $N^g$ is normalized by $Q$, we have $t_{\ell,0}^{-1}t_{h,w}t_{\ell,0}\in N^g$. Then since $N^g$ is abelian, we obtain $t_{\ell,0}^{-1}t_{h,w}t_{\ell,0}t_{h,w}=t_{h,w}t_{\ell,0}^{-1}t_{h,w}t_{\ell,0}$, which implies that
\[
(ac+a)v_1+(ac+bc^{-1}+b)v_2=0^{t_{\ell,0}^{-1}t_{h,w}t_{\ell,0}t_{h,w}}=0^{t_{h,w}t_{\ell,0}^{-1}t_{h,w}t_{\ell,0}}=(ac+a)v_1+(ac^{-2}+bc^{-1}+b)v_2
\]
by direct calculation. Since $c\neq c^{-2}$, this yields $a=0$, that is, $w\in U$. Consequently, $t_{1,w}\in\{t_{1,u}\mid u\in U\}=N\cap N^g$, and so $t_{h,0}=t_{h,w}t_{1,w}^{-1}\in N^g$, contradicting the fact that $N^g$ is regular on $V$.

\underline{Case~2}: $(G_{\a\b})_0$ is the stabilizer in $(G_\a)_0$ of a $1$-dimensional subspace of $V$. In this case, $PQ=\SL(V)_{\l v_1\r}$ for some $v_1\in V$. Extend $v_1$ to a basis $v_1,v_2,\dots,v_e$ of $V$. Then the linear transformation $h$ defined by
\[
v_2^h=v_1+v_2\quad\text{and}\quad v_i^h=v_i\ \text{ for }\ i\in\{1,3,4,\dots,e\}
\]
lies in $P$, and so there exists $w\in V$ such that $t_{h,w}\in N^g$. Then for any $\ell\in(PQ)_w$, since
\[
t_{h\ell h^{-1}\ell^{-1},0}=t_{h,w}t_{\ell h\ell^{-1},w}^{-1}=t_{h,w}(t_{\ell,0}t_{h,w}^{-1}t_{\ell,0}^{-1})\in N^g
\]
and $N^g$ acts regularly on $V$, we obtain $h\ell h^{-1}\ell^{-1}=1$, that is, $\ell\in\Cen_{PQ}(h)$. Hence
\begin{equation}\label{Eqn5}
(PQ)_w\le\Cen_{PQ}(h).
\end{equation}
If there exists $z\in\Fix(h)\setminus\l v_1,v_2,w\r$, then the linear transformation $\ell$ such that $v_1^\ell=v_1$, $w^\ell=w$ and $z^\ell=v_2+z$ satisfies $\ell\in\SL(V)_{\l v_1\r}=PQ$, $w^\ell=w$ and $z^{h\ell}=v_2+z\ne v_1+v_2+z=z^{\ell h}$, contradicting~\eqref{Eqn5}. Thus $\Fix(h)\le\l v_1,v_2,w\r$. As $v_2\notin\Fix(h)$, we obtain
\begin{equation}\label{Eqn6}
\Fix(h)<\l v_1,v_2,w\r.
\end{equation}
In particular, $e-1=\dim(\Fix(h))<3$, which means $e<4$. The possibility of $e=2$ has already been ruled out in Case~1. Therefore, $e=3$. Then it follows from $(e,p^f)\notin\{(3,2),(3,3)\}$ that $p^f\ge4$. Hence there exists $a\in\bbF_{p^f}^\times$ with $a\neq a^{-1}$.
Since $\Fix(h)=\l v_1,v_3\r$, we derive from~\eqref{Eqn6} that $w\notin\l v_1\r$ and $v_2\notin\l v_1,w\r$. Let $\ell$ be the transformation defined by $v_1^k=av_1$, $w^k=w$ and $v_2^k=a^{-1}v_2$. Then $\ell\in\SL(V)_{\l v_1\r}=PQ$, $w^\ell=w$, and $v_2^{h\ell}=av_1+a^{-1}v_2\neq a^{-1}v_1+a^{-1}v_2=v_2^{\ell h}$. This contradicts~\eqref{Eqn5}.
\end{proof}

\begin{lemma}\label{k>2}
Let $L=\A_{n+k}$ with $1\le k\le5$. If $L_\a$ is $k$-transitive on $n+k$ points, then $k\le 2$.
\end{lemma}

\begin{proof}
Let $\Ome=\{1,\dots,n+k\}$ be the set that $L=\A_{n+k}$ naturally acts on.
Suppose for a contradiction that $k\ge3$.
Then $L_\a$ is a $3$-transitive subgroup of $\A_{n+k}$, and so by~\cite[Table~7.4]{Cameron}, the triple $(L_\a,L_{\a1},n+k)$ lies in Table~\ref{3-Trans},
where $L_{\a1}$ is the stabilizer of the point $1\in\Ome$ in $L_{\a}$.
In particular, $L_\a$ is nonsolvable.

\begin{table}[ht]
\[
\begin{array}{lllll}
\hline
\text{Row}& L_\a & L_{\a1} & n+k &\text{Remarks}\\
\hline
1&\AGL_d(2)& \GL_d(2) & 2^d & d\ge3\\
2&\M_{11}& \M_{10} & 11 & \text{sharply $4$-transitive}\\
3&\M_{11}& \PSL_2(11) & 12 & \\
4&\M_{12}& \M_{11} & 12 & \text{sharply $5$-transitive}\\
5&\M_{22}& \PSL_3(4) & 22 & \\
6&\M_{23}& \M_{22} & 23 & \text{$4$-transitive}\\
7&\M_{24}& \M_{23} & 24 & \text{$5$-transitive}\\
8&\PSL_2(p^f).\calO & \Big(p^f{:}\frac{p^f-1}{(2,p-1)}\Big).\calO & p^f+1 & p\text{ prime, }(2,p-1)\le\calO\le(2,p-1)\times f,\\
&&&& \calO\text{ is not Frobenius if }p>2 \\
\hline
\end{array}
\]
\caption{$3$-transitive permutation groups on $n+k$ points}\label{3-Trans}
\end{table}

Since $L_\a$ is $2$-transitive on $\Ome$, it follows from Lemma~\ref{Normalizer} that
\begin{equation}\label{Eqn1}
\Nor_L(L_{\a1})=\Nor_{L_1}(L_{\a1}),
\end{equation}
where $L_1$ is the stabilizer of $1\in\Ome$ in $L$.
Take $\b\in\Ga(\a)$.
As $L_\a$ is nonsolvable, Lemma~\ref{(L,2)-ArcT} asserts that $\Ga$ is $(L,2)$-arc-transitive.
Hence $L_\a$ is $2$-transitive on $\Ga(\a)$, which can be identified with $[L_\a:L_{\a\b}]$.
In particular, $L$ has a suborbit of length $|L_\a|/|L_{\a\b}|=\val(\Ga)$ on $[L:L_\a]$.
Let $M=\Nor_L(L_{\a\b})$.
Then by Lemma~\ref{Sab64} we have $\l L_\a,M\r=L$.
In the following we consider each row of Table~\ref{3-Trans} separately.

\vskip0.1in
\noindent{\underline{Row 1.}}{\hspace{5pt}}
For this row we have $L=\A_{2^d}$, $L_\a=\AGL_d(2)$ and $L_{\a1}=\GL_d(2)\cong\PSL_d(2)$, where $d\ge3$.
Then Lemma~\ref{AffineFaithful} asserts that $L_\a$ acts faithfully on $\Ga(\a)$.
Since $L_\a=\AGL_d(2)$ has a unique faithful $2$-transitive permutation presentation,
it follows that $L_{\a\b}$ is conjugate to $L_{\a1}$ in $L_\a$, that is, $L_{\a\b}=(L_{\a1})^x$ for some $x\in L_\a$.
Since $L_{\a1}\cong\PSL_d(2)$ is $2$-transitive on $\Ome\setminus\{1\}$
and $\PSL_d(2)$ is the only $2$-transitive permutation group with socle $\PSL_d(2)$, we have $\Nor_{L_1}(L_{\a1})=L_{\a1}$.
This together with~\eqref{Eqn1} leads to $\Nor_L(L_{\a1})=L_{\a1}$. Hence
\[
M=\Nor_L(L_{\a\b})=\Nor_L((L_{\a1})^x)=(\Nor_L(L_{\a1}))^x=(L_{\a1})^x\le L_\a,
\]
which implies that $\l L_\a,M\r=L_\a<L$, a contradiction.

\vskip0.1in
\noindent{\underline{Row 2.}}{\hspace{5pt}}
Then $L=\A_{11}$ and $L_\a=\M_{11}$.
Since the $2$-transitive permutation representations of $L_\a=\M_{11}$ have degree $11$ or $12$, we have $\val(\Ga)=11$ or $12$.
However, computation in \magma~\cite{Magma} shows that $L=\A_{11}$ acting on $[\A_{11}:\M_{11}]$ has no suborbit of length $11$ or $12$, a contradiction.

\vskip0.1in
\noindent{\underline{Row 3.}}{\hspace{5pt}}
Then $L=\A_{12}$ and $L_\a=\M_{11}$.
For the same reason as Row~2, we have $\val(\Ga)=11$ or $12$.
Then since computation in \magma~\cite{Magma} shows that $L=\A_{12}$ has no suborbit of length $12$ on $[\A_{12}:\M_{11}]$,
we obtain $|L_\a|/|L_{\a\b}|=\val(\Ga)=11$.
Thus $L_{\a\b}=\M_{10}$.
It then follows from the Atlas~\cite{Atlas} that $M=\Nor_L(L_{\a\b})=\M_{10}.2$.
However, this yields $\l L_\a,M\r=\M_{12}$, contradicting the condition $\l L_\a,M\r=L$.

\vskip0.1in
\noindent{\underline{Row 4.}}{\hspace{5pt}}
For this row we have $L=\A_{12}$ and $L_{\a}=\M_{12}$.
Then $|L_\a|/|L_{\a\b}|=\val(\Ga)=12$ as the $2$-transitive permutation representations of $\M_{12}$ have degree $12$.
However, computation in \magma~\cite{Magma} shows that $L=\A_{12}$ has no suborbit of length $12$ on $[\A_{12}:\M_{12}]$, a contradiction.

\vskip0.1in
\noindent{\underline{Row 5.}}{\hspace{5pt}}
Then $L=\A_{22}$ and $L_\a=\M_{22}$.
Since the unique $2$-transitive permutation representation of $L_\a=\M_{22}$ is of degree $22$, we have $\val(\Ga)=22$ and $L_{\a\b}=\PSL_3(4)$.
Then computation in \magma~\cite{Magma} shows that $M=\Nor_L(L_{\a\b})=L_{\a\b}.3$.
This implies that every $2$-element $g$ of $M$ lies in $L_{\a\b}$ and hence $\l L_\a,g\r=L_\a<L$, contradicting Lemma~\ref{Sab64}.

\vskip0.1in
\noindent{\underline{Row 6.}}{\hspace{5pt}}
Then $L=\A_{23}$ and $L_\a=\M_{23}$.
Since the unique $2$-transitive permutation representation of $L_\a=\M_{23}$ is of degree $23$, we have $\val(\Ga)=23$ and $L_{\a\b}=\M_{22}$.
Then computation in \magma~\cite{Magma} shows that $M=\Nor_L(L_{\a\b})=L_{\a\b}$, which implies $\l L_\a,M\r=L_\a<L$, a contradiction.

\vskip0.1in
\noindent{\underline{Row 7.}}{\hspace{5pt}}
Then $L=\A_{24}$ and $L_{\a}=\M_{24}$.
Similarly as Row 6, we derive that $\val(\Ga)=24$ and $L_{\a\b}=\M_{23}$,
and then computation in \magma~\cite{Magma} shows $M=L_{\a\b}$,
so $\l L_\a,M\r=L_\a<L$, giving a contradiction.

\vskip0.1in
\noindent{\underline{Row 8.}}{\hspace{5pt}}
Then $L_\a$ is an almost simple $3$-transitive group on $\Ome$ with socle $\PSL_2(p^f)$.
Since $L_\a$ is $2$-transitive on $[L_\a:L_{\a\b}]$, the classification of $2$-transitive groups shows that
either $L_{\a\b}$ is conjugate to $L_{\a1}$ in $L_\a$, or $(L_\a,L_{\a\b})=(\PSL_2(11),\A_5)$ or $(\PGaL_2(8),9{:}6)$.
As $L_\a$ is $3$-transitive group on $\Ome$, we have $L_\a\ne\PSL_2(11)$.
If $(L_\a,L_{\a\b})=(\PGaL_2(8),9{:}6)$, then computation in \magma~\cite{Magma} shows that $M=\Nor_L(L_{\a\b})=L_{\a\b}$,
which implies $\l L_\a,M\r=L_\a<L$, a contradiction.

Thus $L_{\a\b}$ is conjugate to $L_{\a1}$ in $L_\a$, that is, $L_{\a\b}=(L_{\a1})^x$ for some $x\in L_\a$.
Let $K=\PGaL_2(p^f)$ be an overgroup of $L_\a$ in $\Sym(\Ome)=\Sy_{n+k}$.
Then the stabilizer $K_1=\AGaL_1(p^f)$ is $2$-transitive on $\Ome\setminus\{1\}$.
Also, $L_{\a1}$ is $2$-transitive on $\Ome\setminus\{1\}$, and so is its overgroup $\Nor_{L_1}(L_{\a1})$.
Now $\Nor_{L_1}(L_{\a1})$ is a $2$-transitive group on $\Ome\setminus\{1\}$
and has a normal $2$-transitive subgroup $L_{\a1}$ on $\Ome\setminus\{1\}$ with $L_{\a1}\le K_1=\AGaL_1(p^f)$.
We conclude from the classification of $2$-transitive groups that $\Nor_{L_1}(L_{\a1})\le K_1=\AGaL_1(p^f)$.
This together with~\eqref{Eqn1} implies that $(\Nor_L(L_{\a1}))^x=(\Nor_{L_1}(L_{\a1}))^x\le(K_1)^x\le K$ as $x\in L_\a\le K$.
Hence
\[
\l L_\a,M\r=\l L_\a,\Nor_L(L_{\a\b})\r=\l L_\a,\Nor_L((L_{\a1})^x)\r=\l L_\a,(\Nor_L(L_{\a1}))^x\r\le K,
\]
contradicting the condition $\l L_\a,M\r=L$.
\end{proof}

By the above lemma, we have either $k=1$ or $k=2$. If $k=1$ then part (b) of Proposition \ref{Quasi-case} holds, and for $k=2$, to complete the
statement of part (c) of Proposition \ref{Quasi-case}, we remain to show the following:

\begin{lemma}
Let $L=\A_{n+2}$. If $L_\a$ is $2$-transitive on $n+2$ points, then $\Ga$ is $(L,2)$-arc-transitive.
\end{lemma}
\begin{proof}
Recall $L_\a$ is either almost simple or affine.  The lemma is true for  the case $L_\a$ is almost simple
by Lemma~\ref{(L,2)-ArcT}. So we assume that $L_\a$ is an affine 2-transitive subgroup of $\A_{n+2}$. Then $G_\a$ is an affine 2-transitive subgroup of $\Sy_{n+2}$. By Lemma \ref{AffineFaithful}, $G_\a$ acts faithfully on $\Ga(\a)$. Thus $G_\a^{\Ga(\a)}$ is equivalent to the  $2$-transitive permutation representation of $G_\a$ naturally on $n+2$ points. This implies $L_\a^{\Ga(\a)}$ is 2-transitive as $L_\a$ is 2-transitive on $n+2$ points. Thus $\Ga$ is $(L,2)$-arc-transitive.
\end{proof}

\subsection{Case~(a.3)}

Next we consider $(L, n, L_\a)$ in Table~\ref{Sporadic}, as in case~(a.3).
We deal with Row~21 of Table~\ref{Sporadic} separately in the following lemma.

\begin{lemma} \label{Row21}
The triple $(L,n,L_\a)$ cannot lie in Row~$21$ of Table~$\ref{Sporadic}$.
\end{lemma}

\begin{proof}
Suppose for a contradiction that $L=\Ome_8^+(2)$, $n=9$ and $2^4{:}\A_5\leq L_\a\leq 2^6{:}\A_8$, as in Row~21 of Table~\ref{Sporadic}.
Then computation in \magma~\cite{Magma} shows that there are 36 candidates for $L_\a$ (up to $L$-conjugate) giving rise to a factorization $L=\A_nL_\a$.
In Table~\ref{tb-La}, we list these candidates together with the lengths of their orbits on $V(\Ga)\setminus\{\a\}$, where each row corresponds to two isomorphic conjugacy classes, and the expression $64^170^1$ (in the last row) etc. means that
there are exactly 1 orbit with length 64 and 1 orbit with length 70 of $L_\a$ in the corresponding row etc.

\begin{table}[ht]
\[
\begin{array}{ll}
\hline
L_\a & \text{Orbits on } V(\Ga)\setminus\{\a\}\\
\hline
2^4{:}\A_5 & 1^{23}10^{132}16^{96}80^{441}160^{72}240^{1}320^{216}480^{98}960^{16} \\
2^4{:}\Sy_5 & 1^{1}2^{5}10^{6}16^{4}20^{30}32^{22}40^{1}80^{24}120^{1}160^{114}240^{6}320^{32}480^{22}640^{43}960^{18}1920^{1}\\
2^4{:}3{:}\A_5	& 1^{1}3^{2}10^{2}16^{2}30^{14}48^{10}80^{1}160^{2}240^{49}320^{1}480^{10}960^{26}1440^{10}2880^{1}\\
2^4{:}3{:}\Sy_5 & 3^{1}10^{1}16^{1}30^{5}40^{1}48^{3}60^{1}96^{1}120^{1}160^{2}240^{16}480^{15}720^{4}960^{4}1440^{4}1920^{3}\\
2^5{:}\A_5 & 1^{3}2^{4}5^{4}10^{16}20^{24}32^{24}40^{9}80^{24}120^{1}160^{100}240^{2}320^{20}480^{16}640^{52}960^{16}1920^{4}\\
2^5{:}\Sy_5 & 1^{1}4^{1}5^{2}20^{9}32^{4}40^{14}60^{1}64^{4}80^{5}160^{12}240^{9}320^{28}480^{4}640^{15}960^{8}1280^{6}1920^{2} \\
2^5{}^{\boldsymbol{.}}\Sy_5 & 1^{1}4^{1}5^{2}20^{5}40^{8}60^{1}64^{6}80^{1}160^{6}240^{1}320^{27}480^{2}640^{4}960^{3}1280^{13}1920^{4}3840^{1} \\
 2^6 {:}\A_5 & 1^{5}5^{18}20^{9}40^{6}60^{1}64^{6}120^{2}160^{18}320^{27}480^{6}640^{9}960^{9}1280^{9}3840^{1}\\
 \text{[}2^6 \text{]}{:}\Sy_5 & 2^{1}5^{1}10^{5}20^{2}30^{1}40^{2}64^{1}80^{3}120^{1}128^{1}160^{5}240^{2}320^{6}480^{3}640^{8}960^{2}1280^{3}1920^{2}2560^{1}\\
2^6{:}3{:}\A_5 &1^{1}15^{2}20^{1}40^{4}60^{1}64^{2}480^{4}960^{3}1280^{1}1920^{1}2880^{1}3840^{1}\\
2^6{:}3{:}\Sy_5 & 10^{1}15^{1}30^{1}40^{2}64^{1}240^{4}480^{1}640^{1}960^{2}1440^{1}1920^{1}\\
\text{[}2^5\text{]}{:}\A_6 & 1^{3}12^{4}32^{8}60^{5}120^{4}192^{16}480^{8}640^{4}720^{1}960^{4}\\
2^5{:}\Sy_6 & 1^{1}24^{1}30^{5}32^{4}240^{1}360^{1}384^{4}480^{8}640^{2}\\
\text{[}2^5\text{]}{}^{\boldsymbol{.}}\Sy_6 & 1^{1}24^{1}30^{1}64^{2}120^{3}360^{1}384^{4}640^{2}960^{2}1920^{1}\\
\text{[}2^6\text{]}{:}\A_6 & 1^{1}6^{4}30^{1}64^{2}120^{1}240^{1}360^{1}384^{4}480^{4}640^{2}1920^{1}\\
\text{[}2^6\text{]}{:}\Sy_6 & 12^{1}15^{1}60^{1}64^{1}120^{1}180^{1}480^{2}640^{1}768^{1}960^{1}\\
2^6{:}\A_7 & 7^{1}64^{1}280^{2}448^{1}\\
2^6{:}\A_8 & 64^{1}70^{1}\\
\hline
\end{array}
\]
\caption{Candidates for $L_\a$ in Row~$21$ of Table~\ref{Sporadic}}\label{tb-La}
\end{table}

Since $L_\a$ is insolvable, Lemma~\ref{(L,2)-ArcT} asserts that $\Ga$ is $(L,2)$-arc-transitive, and so $L_\a^{\Ga(\a)}$ is $2$-transitive.
Thus $L_\a^{\Ga(\a)}$ is affine or almost simple.

First assume that $L_\a^{\Ga(\a)}$ is affine.
Inspecting Table~\ref{tb-La} we see that the homomorphic image $L_\a^{\Ga(\a)}$ of $L_\a$ has socle a $2$-group, and has a unique nonsolvable composition factor, which is one of $\A_5$, $\A_6$, $\A_7$ or $\A_8$.
By the classification of $2$-transitive groups, a $2$-transitive affine group with socle a $2$-group and the nonsolvable composition factor $\A_5$, $\A_6$, $\A_7$ or $\A_8$ is one of the following:
\[
2^4{:}\A_6,\ 2^4{:}\Sy_6,  \ 2^4{:}\A_7,\ \ 2^4{:}\SL_2(4),\ 2^4{:}\GL_2(4), \ 2^4{:}\GaL_2(4),\ \ 2^4{:}\SL_4(2),\ \ 2^4{:}\Sp_4(2).
\]
Since these permutation groups all have degree $16$, we conclude that $|\Ga(\a)|=16$.
Thus $L_\a$ has an orbit of length $16$.
Then by Table~\ref{tb-La}, the possibilities for $L_\a$ are:
\[
2^4{:}\A_5,\ \ 2^4{:}\Sy_5,\ \ 2^4{:}3{:}\A_5,\ \ 2^4{:}3{:}\Sy_5.
\]
However, for each possible $L_\a$ and each orbit of length $16$, computation in \magma~\cite{Magma} shows that the action of $L_\a$ on this orbit is not $2$-transitive, contradicting the $2$-transitivity of $L_\a^{\Ga(\a)}$.

Next assume that $L_\a^{\Ga(\a)}$ is almost simple.
Then from Table~\ref{tb-La} we see that $\Soc(L_\a^{\Ga(\a)})=\A_t$ with $t\in\{5,6,7,8\}$.
By the classification of almost simple $2$-transitive groups, either $|\Ga(\a)|=t$ or $(t,|\Ga(\a)|)\in\{(5,6),(7,15),(8,15)\}$.
As $L_\a$ has an orbit of length $|\Ga(\a)|$, by checking Table~\ref{tb-La} we obtain the following possibilities for $(L_\a,|\Ga(\a)|)$:
\[
(2^5{:}\A_5,5)\ \ (2^5{:}\Sy_5,5),\ \ (2^5{}^{\boldsymbol{.}}\Sy_5,5)\ \ (2^6{:}\A_5,5)\ \ ([2^6]{:}\Sy_5,5),\ \ ([2^6]{:}\A_6,6),\ \ (2^6{:}\A_7,7).
\]
However, for each possible $L_\a$ in these pairs of $(L_\a,|\Ga(\a)|)$ and for each orbit of length $|\Ga(\a)|$, computation in \magma~\cite{Magma} shows that the action of $L_\a$ on this orbit is not $2$-transitive, contradicting the $2$-transitivity of $L_\a^{\Ga(\a)}$.
\end{proof}

We are now in a position to completely determine $(G,H,\Ga)$ for case~(a.3), which will finish the proof of Proposition~\ref{Quasi-case}.

\begin{lemma}\label{Small-gps}
Let $(L,n,L_\a)$ be in Table~$\ref{Sporadic}$.
Then $n=5$, and one of the following holds:
\begin{enumerate}[{\rm (a)}]
\item $G=\M_{12}$, $G_\a=\M_{11}$ and $\Ga=\K_{12}$;
\item $G=\PGL_2(p)$ with $p\in\{11,19,29,59\}$, $G_\a=\AGL_1(p)$ and $\Ga=\K_{p+1}$.
\end{enumerate}
\end{lemma}

\begin{proof}
By Lemma~\ref{Row21}, Row 21 of Table~\ref{Sporadic} cannot occur.
Thus we only need to deal with Rows~1--20 and~22--24 of Table~\ref{Sporadic}.

\vskip0.1in
\noindent{\underline{Row 1.}}{\hspace{5pt}}
For this row we have $L=\A_6$, $n=5$ and $L_{\a}=\A_4$ or $\Sy_4$. Then $\A_6\cong\PSL_2(9)\le G\le\PGaL_2(9)$, and $G=HG_\a$ with $H=\A_5$ and $G_\a\cap L=\A_{4}$ or $\Sy_4$. Searching for such a group factorization in \magma~\cite{Magma} shows that $(G,G_\a)$ is the one of the following pairs:
\[
(\A_6,\,\A_4),\ \ (\A_6,\,\Sy_4),\ \ (\Sy_6,\,\Sy_4),\ \ (\Sy_6,\,\A_4\times2),\ \ (\Sy_6,\,\Sy_4\times2).
\]
By Lemma \ref{lm-AnAn-2Sn-2}, the first two candidates for $(G,G_\a)$ are not possible.
Hence $G=\Sy_6$ and $G_\a=\Sy_4$, $\A_4\times2$ or $\Sy_4\times2$.
Then since the action of $G_\a$ on $[G_\a:G_{\a\b}]$ is $2$-transitive, we conclude that $|G_\a|/|G_{\a\b}|=3$ or $4$.
For the three possibilities of $G_\a$, computation in \magma~\cite{Magma} shows that $G=\Sy_6$ has a suborbit of length $3$ or $4$ only if $G_\a=\Sy_4$ and the suborbit length is $4$.
However, in this case computation in \magma~\cite{Magma} shows that there is no element $g\in G$ satisfying:
\[
g\in\Nor_G(G_\a\cap G_\a^g),\ \ g^2\in G_\a,\ \ \l G_\a,g\r=G,\ \ |G_\a|/|G_\a\cap G_\a^g|=4.
\]
This contradicts Lemma~\ref{Sab64}.

\vskip0.1in
\noindent{\underline{Rows 2--3.}}{\hspace{5pt}}
For these two rows we have $L=\A_m$ with $m\in\{10,15\}$ and $L_{\a}=\A_{m-2}$ or $\Sy_{m-2}$.
Since $L_\a$ is insolvable, Lemma~\ref{(L,2)-ArcT} asserts that $\Ga$ is $(L,2)$-arc-transitive.
This contradicts Lemma~\ref{lm-AnAn-2Sn-2}.

\vskip0.1in
\noindent{\underline{Row 4.}}{\hspace{5pt}}
Here $L=\M_{12}$, $n=5$ and $L_\a=\M_{11}$.
Then $L$ is 2-transitive on $V(\Ga)|$ with $|V(\Ga)|=|L|/|L_\a|=12$,
it follows that $\Ga=\K_{12}$.
According to~\cite{Atlas}, $\Out(\M_{12})=2$, and $\Aut(\M_{12})$ has no subgroup of index $12$.
Hence $G\ne\Aut(\M_{12})$, and so $G=\M_{12}$.
This leads to part~(a) of the lemma.


\vskip0.1in
\noindent{\underline{Rows 5--8.}}{\hspace{5pt}}
Here $L=\PSL_2(p)$ with $p\in\{11,19,29,59\}$, $n=5$ and $p\le L_\a\le p{:}(p-1)/2$.
Since $G_\a$ is $2$-transitive on $\Ga(\a)$, we have $G_\a=\AGL_1(p)=p{:}(p-1)$ and $G=\PGL_2(p)$.
Hence $|\Ga(\a)|=p$ and $|V(\Ga)|=|G|/|G_\a|=|\PGL_2(p)|/|\AGL_1(p)|=p+1$.
This implies that $\Ga\cong\K_{p+1}$, leading to part~(b) of the lemma.

\vskip0.1in
\noindent{\underline{Row 9.}}{\hspace{5pt}}
For this row, $L=\PSL_{4}(3)$ and $L_\a=3^3{:}\PSL_3(3)$.
Then $L$ and $G$ are $2$-transitive but not 3-transitive on $V(\Ga)$,
which implies that $\Ga$ is a complete graph.
Notice that an automorphism group acting 2-arc-transitively on a complete graph
should be 3-transitive on the vertices,  contradicting that $G$ is not 3-transitive on $V(\Ga)$.


\vskip0.1in
\noindent{\underline{Row 10.}}{\hspace{5pt}}
For this row, $L=\PSU_{3}(5)$ and $L_{\a}=5_+^{1+2}{:}8$.
Then $L$ and $G$ are $2$-transitive but not 3-transitive on $V(\Ga)$,
hence $\Ga$ is a complete graph, a contradiction occurs by the same reason as in Row 9.


\vskip0.1in
\noindent{\underline{Row 11.}}{\hspace{5pt}}
Here $L=\PSp_{4}(3)$ and $L_{\a}=3_+^{1+2}{:}\Q_8$ or $3_+^{1+2}{:}2.\A_4$.
Since $\Out(L)=2$, we have $G_\a=L_\a$ or $L_\a.2$.
Then since $G_\a$ is $2$-transitive on $\Ga(\a)$, we see that $(G,G_\a,|\Ga(\a)|)$ is one of the following triples:
\[
(\PSp_{4}(3),\,3_+^{1+2}{:}\Q_8,\,9),\ \ (\PSp_{4}(3),\,3_+^{1+2}{:}2.\A_4,\,4),
\]
\[
(\PSp_{4}(3).2,\,3_+^{1+2}{:}\Q_8.2,\,9),\ \ (\PSp_{4}(3).2,\,3_+^{1+2}{:}2.\A_4.2,\,4).
\]
However, for each of the possible triples $(G,G_\a,|\Ga(\a)|)$, computation in \magma~\cite{Magma} shows that $G$ acting on $[G:G_\a]$ has no suborbit of length $|\Ga(\a)|$, a contradiction.

\vskip0.1in
\noindent{\underline{Rows 12--13.}}{\hspace{5pt}}
For these two rows, $L=\Sp_{6}(2)$, and $L_{\a}$ is one of the following groups:
\[
\PSU_3(3){:}2,\ \ 3_+^{1+2}{:}8{:}2,\ \ 3_+^{1+2}{:}2{.}\Sy_4,\ \ \PSL_2(8),\ \ \PSL_2(8){:}3,\ \ \PSU_4(2){:}2.
\]
Since $\Out(L)=1$, we have $G=L$, and so $\Ga$ is $(L,2)$-arc-transitive.
Then since $G_\a$ is $2$-transitive on $\Ga(\a)$, we see that $(G_\a,|\Ga(\a)|)$ is one of the following pairs:
\[
(\PSU_3(3){:}2,\,28),\ \ (3_+^{1+2}{:}8{:}2,\,9),\ \ (3_+^{1+2}{:}2{.}\Sy_4,\,9),\ \ (3_+^{1+2}{:}2{.}\Sy_4,\,4),
\]
\[
(3_+^{1+2}{:}2{.}\Sy_4,\,3),\ \ (\PSL_2(8),\,9),\ \ (\PSL_2(8){:}3,\,9),\ \ (\PSL_2(8){:}3,\,28).
\]
However, for each of the possible pairs $(G_\a,|\Ga(\a)|)$, computation in \magma~\cite{Magma} shows that $G$ acting on $[G:G_\a]$ has no suborbit of length $|\Ga(\a)|$, a contradiction.

\vskip0.1in
\noindent{\underline{Row 14.}}{\hspace{5pt}}
For this row, $L_\a=\mathrm{SO}_8^-(2)$ is nonsolvable, and so $\Ga$ is $(L,2)$-arc-transitive by Lemma~\ref{(L,2)-ArcT}.
However, $\mathrm{SO}_8^-(2)$ has no $2$-transitive permutation representation, a contradiction.

\vskip0.1in
\noindent{\underline{Rows 15--16.}}{\hspace{5pt}}
For these two rows, $L=\Ome_7(3)$, and $L_\a$ is one of the following groups:
\[
3^{3+3}{:}\PSL_3(3),\ \ 3^3{:}\PSL_3(3),\ \ \PSL_4(3),\ \ \PSL_4(3){:}2,\ \ \G_2(3).
\]
In particular, $L_\a$ is nonsolvable.
Hence Lemma~\ref{(L,2)-ArcT} implies that $\Ga$ is $(L,2)$-arc-transitive.
Then as $L_\a$ is $2$-transitive on $\Ga(\a)$,
we see that $(L_\a,|\Ga(\a)|)$ is one of the following pairs:
\[
(3^{3+3}{:}\PSL_3(3),\,27),\ \ (3^{3+3}{:}\PSL_3(3),\,13),\ \ (3^3{:}\PSL_3(3),\,27),
\]
\[
(3^3{:}\PSL_3(3),\,13),\ \ (\PSL_4(3),\,40),\ \ (\PSL_4(3){:}2,\,40).
\]
However, for each of the possible pairs $(L_\a,|\Ga(\a)|)$, computation in \magma~\cite{Magma} shows that there is no subgroup $K$ of $L_\a$ with $|L_\a|/|K|=|\Ga(\a)|$ and $\l L_\a,\Nor_L(K)\r=L$, contradicting Lemma~\ref{Sab64}.

\vskip0.1in
\noindent{\underline{Rows 17--20.}}{\hspace{5pt}}
Here $L=\Ome_8^+(2)$, and $L_\a$ is one of the following groups:
\[
\Sp_6(2),\ \ \A_9,\ \ \PSU_4(2),\ \ \PSU_4(2){:}2,\ \ 3\times\PSU_4(2),\ \ (3\times\PSU_4(2)){:}2,
\]
\[
2^4{:}15.4,\ \ 2^6{:}15,\ \ 2^6{:}15.2,\ \ 2^6{:}15.4,\ \ \A_8,\ \ \Sy_8.
\]

Suppose that $L_\a$ is nonsolvable.
Then $\Ga$ is $(L,2)$-arc-transitive by Lemma~\ref{(L,2)-ArcT}, and so $L_\a$ is $2$-transitive on $\Ga(\a)$,
and $L_\a^{\Ga(\a)}$ is almost simple by Lemma~\ref{Insoluble}, which implies that $(L_\a,|\Ga(\a)|)$ is one of the following pairs:
\[
(\Sp_6(2),\,28),\ \ (\Sp_6(2),\,36),\ \ (\A_9,\,9),\ \ (\A_8,\,8),\ \ (\A_8,\,15),\ \ (\Sy_8,\,8).
\]
However, for each of these pairs $(L_\a,|\Ga(\a)|)$, computation in \magma~\cite{Magma} shows that $L$ acting on $[L:L_\a]$ has no suborbit of length $|\Ga(\a)|$, a contradiction.
Thus $L_\a$ is solvable, that is, $L_\a=2^4{:}15.4$, $2^6{:}15$, $2^6{:}15.2$ or $2^6{:}15.4$.
These candidates for $L_\a$ only occur in Row~20 of Table~\ref{Sporadic}, where it is shown that $n=9$.
As $L\le G\le\Aut(L)$ and $G_\a\cap L=L_\a\in\{2^4{:}15.4,2^6{:}15,2^6{:}15.2,2^6{:}15.4\}$, computation in \magma~\cite{Magma} shows that such pairs $(G,G_\a)$ that give rise to a factorization $G=\A_9G_\a$ are the following:
\[
(\Ome_8^+(2),\,2^4{:}15.4),\ \ (\Ome_8^+(2),\,2^6{:}15),\ \ (\Ome_8^+(2),\,2^6{:}15.2),\ \ (\Ome_8^+(2),\,2^6{:}15.4),
\]
\[
(\GO_8^+(2),\,(2^4{:}15.4){:}2),\ \ (\GO_8^+(2),\, (2^6{:}15){:}2),\ \ (\GO_8^+(2),\,(2^6{:}15.2).2),\ \ (\GO_8^+(2),\,(2^6{:}15.4){:}2).
\]
Then as $G_\a$ is $2$-transitive on $\Ga(\a)$, we conclude that $|\Ga(\a)|=16$.
However, for each of the above pairs $(G,G_\a)$, computation in \magma~\cite{Magma} shows that $G$ acting on $[G:G_\a]$ has no suborbit of length $16$ on which $G_\a$ acts $2$-transitively, a contradiction.

\vskip0.1in
\noindent{\underline{Row 22.}}{\hspace{5pt}}
For this row, $L=\POm_8^{+}(3)$, and $L_\a=3^6{:}\PSL_4(3)$ is nonsolvable.
Hence $\Ga$ is $(L,2)$-arc-transitive by Lemma~\ref{(L,2)-ArcT}.
This implies that $L_\a$ is $2$-transitive on $\Ga(\a)$, and so $|\Ga(\a)|=40$.
However, computation in \magma~\cite{Magma} shows that $L=\POm_8^{+}(3)$ acting on $[L:L_\a]$ has no suborbit of length $40$, a contradiction.

\vskip0.1in
\noindent{\underline{Rows 23--24.}}{\hspace{5pt}}
Here $L_\a=\Omega_7(3)$ or $2^8{:}\Omega_{8}^-(2)$.
Then Lemma~\ref{(L,2)-ArcT} asserts that $\Ga$ is $(L,2)$-arc-transitive since $L_\a$ is nonsolvable.
However, neither $\Omega_7(3)$ nor $2^8{:}\Omega_{8}^-(2)$ has a $2$-transitive permutation representation, a contradiction.
\end{proof}

\section{Proof of Theorem~\ref{th:1}}\label{sec2}

Let $\Ga$ be a connected $(G,2)$-arc-transitive Cayley graph on an alternating group $H\cong\A_n$ with $H\leq G$ and $n\ge 5$, and let $\a$ be a vertex of $\Ga$. Let $M$ be a maximal intransitive normal subgroup of $G$. Suppose that $H$ is not normal in $G$. We divides the proof into several lemmas.

\begin{lemma}
If $M=1$, then part {\rm (a)}, {\rm (b)} or {\rm (c)} of Theorem~$\ref{th:1}$ holds.
\end{lemma}

\begin{proof}
Since $M=1$, the group $G$ is quasiprimitive on $V(\Ga)$, and so $\Ga$ is described in Proposition~\ref{Quasi-case}. 
It follows that $G$ is almost simple, and one of (b)--(f) of Proposition~\ref{Quasi-case} occurs as $H$ is not normal in $G$ by our assumption.
Let $L=\Soc(G)$. Since $H$ is nonabelian simple and $H/(H\cap L)\cong HL/L\leq G/L\leq \Out(L)$ is solvable,
we deduce $H\cap L=H$, that is, $H \leq L$.  
Note that both $L$ and $G$ are transitive on $V(\Ga)$. Since $H$ is regular on $V(\Ga)$, we have $|L|=|H||L_\a|$ and $|G|=|H||G_\a|$.

First assume that case~(b) of Proposition~\ref{Quasi-case} occurs. 
Then $L=\A_{n+1}$, and $L_{\a}$ is a transitive subgroup of $\A_{n+1}$. Since $|L_\a|= |L|/|H| =n+1$, we see that $L_{\a}$ is a regular subgroup of $\A_{n+1}$, as in part (a) of Theorem~\ref{th:1}.

Next assume that case~(c) of Proposition~\ref{Quasi-case} occurs. 
Then $\Ga$ is $(L,2)$-arc-transitive with $L=\A_{n+2}$, and $L_{\a}$ is a $2$-transitive subgroup of $\A_{n+2}$ acting faithfully on $\Ga(\a)$. Since $|L_\a|= |L|/|H|=(n+2)(n+1)$, it follows that $L_{\a}$ is a sharply $2$-transitive subgroup of $\A_{n+2}$.  
By~\cite[Chapter XII: Theorem 9.1]{Huppert82b}, sharply $2$-transitive groups are affine groups with degree a prime power. 
Since $L_{\a}$ acts faithfully on $\Ga(\a)$ and $L_{\a}$ has only one $2$-transitive permutation representation, we derive that $\val(\Ga)=n+2$ is a prime power. Hence part (b) of Theorem~\ref{th:1} holds (with $N=L$).

Now assume that case~(d) of Proposition~\ref{Quasi-case} occurs. 
Then $L=\A_m$, where $m$ is the index of a subgroup of $\A_n$, and $L_\a=\A_{m-1}$. 
It follows that $|\A_{m-1}|=|L_\a|=|L|/|H|=|\A_m|/|\A_n|$, which implies $m=|\A_n|$. 
This lies in part~(c) of Theorem~\ref{th:1}.

Finally, for cases (e) and (f) of Proposition~\ref{Quasi-case}, since $|G|=|H||G_\a|$, computation shows that there is only one possibility, namely, $\Ga=\K_{60}$, $H=\A_5$, $G=\PGL_2(59)$ and $G_\a=\AGL_1(59)$. This also lies in part~(c) of Theorem~\ref{th:1}.
\end{proof}

For the rest of this section we assume that $M \neq 1$. If $M$ has exactly two orbits on $V(\Ga)$, then $\Ga$ is a bipartite graph, and so the stabilizer in $H$ of one part has index $2$ in $H$, which is not possible as $H\cong\A_n$ is simple.
Thus $M$ has at least three orbits on $V(\Ga)$.
Then by Theorem \ref{th:Praeger}, $\Ga_M$ is $(G/M,2)$-arc-transitive, $(G/M)_v \cong G_\a$ for any $v \in V(\Ga_M)$, and $\Ga$ is a normal cover of $\Ga_M$.
Moreover, $M$ is semiregular but not transitive on $V(\Ga)$, and so $|M|$ properly divides $|V(\Ga)|=|H|$.
By the maximality of $M$, $G/M$ is quasiprimitive on $V(\Ga)$.
Since $H $ is regular on $V(\Ga)$, the group $HM/M\cong H\cong\A_n$ is transitive on $V(\Ga_M)$.
Hence the $(G/M,2)$-arc-transitive graph $\Ga_M$ is described in Proposition~\ref{Quasi-case}.

\begin{lemma}
$\Ga_M$ does not satisfy {\rm (e)} or {\rm (f)} of Proposition~$\ref{Quasi-case}$.
\end{lemma}

\begin{proof}
First suppose that $\Ga_M$ satisfies (e) of Proposition~\ref{Quasi-case}. Then $\Ga_M=\K_{p+1}$ and $H=\A_5$, where $p \in \{ 11, 19, 29, 59 \}$. Then $|M|=|V(\Ga)|/|V(\Ga_M)|=60/(p+1)\in \{1,2,3,4,5\}$. By our assumption, $|M| \neq 1$.  By~\cite[Theorem 1.1]{DMW}, $|M|\neq3$ or $5$. If $|M|=2$, then $p=29$ and by \cite[Theorem 1.1]{DMW}, $\Ga=\K_{30,30}-30 \K_2$ is bipartite, which is not possible as $H$ is simple.

Next suppose that $\Ga_M$ satisfies (f) of Proposition~\ref{Quasi-case}. Then $\Ga_M=\K_{12}$ and $H=\A_5$. It follows that $|M|=|V(\Ga)|/|V(\Ga_M)|=60/12=5$ and so $M \cong \C_5$. Note that $\Ga$ is a normal $M$-cover of $\Ga_M=\K_{12}$, and $G/M$ lifts to a $2$-arc-transitive automorphism group $G$ of $\Ga$. We conclude from~\cite[Theorem 1.1]{DMW} that such a graph $\Ga$ dose not exist, a contradiction.
\end{proof}

\begin{lemma}
If $\Ga_M$ satisfies {\rm (a)}, {\rm (b)} or {\rm (d)} of Proposition~$\ref{Quasi-case}$,  then part {\rm (d.1)}, {\rm (d.2)} or {\rm (d.3)} of Theorem~$\ref{th:1}$ holds, respectively.
\end{lemma}

\begin{proof}
It is clear that if $\Ga_M$ satisfies (b) or (d) of Proposition~\ref{Quasi-case} then part~(d.2) or~(d.3) of Theorem~$\ref{th:1}$ holds, respectively. Suppose that $\Ga_M$ satisfies (a) of Proposition~\ref{Quasi-case}. Then $\Soc(G/M)=HM/M\cong\A_n$. 
Let $B=HM=M{:}H$. Then $B/M=\Soc(G/M)$ and so $B \unlhd G$. 
Since $B\geq H$ is transitive on $V(\Ga)$, we have $|M|=|B|/|H|=|B|/|V(\Ga)|=|B_\a|$
and $G=BG_\a$, which implies that $G/B \cong G_\a/B_\a$.
Since $\Ga$ is a normal cover of $\Ga_M$, it follows that $|B_\a|=|(B/M)_v|$ divides $|B/M|=|H|=|\A_n|$, and 
\[
|M|=|B_\a|=|(B/M)_v|=|(\Soc(G/M))_v|.
\]
Recall that $|M|$ properly divides $|H|=|\A_n|$. 
If $B=M \times H$, then $H$ is characteristic in $B$ and hence normal in $G$, a contradiction. Hence $B\neq M \times H$.
Then as in the proof of Lemma \ref{lm:miniM}, there are two normal subgroups $M_j$ and $M_{j+1}$ of $B$ such that 
\[
1\leq M_j<M_{j+1}\leq M,\quad M_jH=M_j\times H,\quad M_{j+1}H \neq M_{j+1} \times H, 
\]
$M_{j+1}/M_j$ is a minimal normal subgroup of $B/M_j$, and $M_jH/M_j\cong\A_n$ acts faithfully on $M_{j+1}/M_j$ by conjugation. 

Suppose that $M_{j+1}/M_j\cong T^m$ for some nonabelian simple group $T$.
Then $2^m\leq|T^m|_2\leq|M|_2<|\A_n|_2<2^n$, and so $m<n$. 
Since $M_jH/M_j\cong\A_n$ acts faithfully on $M_{j+1}/M_j$, we have $\A_n\lesssim\Aut(T^m)=\Aut(T)\wr \Sy_m\cong (T^m.\Out(T)^m){:}\Sy_m$.
As $\Out(T)$ is solvable and $m<n$, we then conclude that $\A_n\le T^m$.
This is a contradiction as $|T^m|\leq|M|<|A_n|$.

Thus we conclude that $M_{j+1}/M_j\cong \ZZ_p^m$ for some prime $p$. Consequently, $\A_n \lesssim \GL_m(p)$.
Since $|M|=|(\Soc(G/M))_v|$, we have $p^m\leq|M|_p\leq|(\Soc(G/M))_v|_p$.
Suppose $5 \leq n \leq 8$.  Recall that $\Ga_M$ is a connected $(G/M,2)$-arc-transitive graph with $\val(\Ga_M)=\val(\Ga)$ and $|(G/M)_v|= |G_\a|$. Inspecting the results in~\cite{XFWX05,DF19,DFZ17} regarding $2$-arc-transitive nonnormal Cayley graphs on nonabelian simple groups of valency $3$, $4$ and $5$, respectively, we conclude that $\val(\Ga) \geq 6$.  
With the help of \magma~\cite{Magma}, we can easily find all the pairs $(G/M,(G/M)_v)$ such that $\Soc(G/M)=\A_n$ with $5\leq n \leq 8$ and the action of $G/M$ on $[G/M:(G/M)_v]$ admits a connected non-bipartite $(G/M,2)$-arc-transitive orbital graph of valency at least $6$. 
Then for each pair $(G/M,(G/M)_v)$, we check whether there exist prime $p$ and integer $m$ such that $p^m\leq|\Soc(G/M)_v|_p$ and $\Soc(G/M)\lesssim\GL_m(p)$. It turns out that no such prime $p$ and integer $n$ exist for any pair $(G/M,(G/M)_v)$, a contradiction.

Therefore, $n\geq 9$. By \cite[Proposition 5.3.7]{K-Lie}, $m\geq n-2$. Note that $(n!)_p<p^{n/(p-1)}$ (see \cite[Example 2.6.1]{DM-book} for example). If $p\geq 3$, then $p^{n-2}\leq p^m\leq|M|_p\leq |\A_n|_p<p^{n/(p-1)}<p^{n-2}$, a contradiction. Hence $p=2$, and $2^{n-2}\leq2^m\leq|M|_2\leq|\A_n|_2$. 
Note that $|\A_n|_2=2^{n-2}$ if $n$ is a $2$-power and $|\A_n|_2<2^{n-2}$ otherwise.
We conclude that $n\geq 16$ is a 2-power, and $|M|_2=|\A_n|_2$.
Thus $|\Soc(G/M)_v|_2=|M|_2=|\A_n|_2=|\Soc(G/M)|_2$, and so $\Soc(G/M)_v$ contains a Sylow $2$-subgroup of $\Soc(G/M)$, leading to part~(d.1) of Theorem~\ref{th:1}.
\end{proof}

\begin{lemma}
If $\Ga_M$ satisfies {\rm (c)} of Proposition~$\ref{Quasi-case}$, then part {\rm (b)} of Theorem~$\ref{th:1}$ holds.
\end{lemma} 

\begin{proof} 
Let $X =G/M$ and $Y =HM/M\cong\A_n$.  Then from Proposition~\ref{Quasi-case}(c) we see that $\Soc(X)\cong\A_{n+2}$, $\Soc(X)_v $ is $2$-transitive on $n+2$ points, the action of $(\Soc(X))_v$ on $\Ga_M(v)$ is faithful, and $\Ga_M$ is $(\Soc(X),2)$-arc-transitive. Since both $X$ and $Y$ are transitive on $V(\Ga_M)$, we have $|X:Y|=|X_v:Y_v|$, and so 
\begin{equation}\label{eq:orderofM}
|M|=\frac{|V(\Ga)|}{|V(\Ga_M)|}=\frac{|Y|}{|V(\Ga_M)|}=|Y_v|=\frac{|X_v|}{|X:Y|}\ \text{ divides }\ \frac{|X_v|}{(n+2)(n+1)}.
\end{equation}  

Suppose that $n\leq8$. As $\Soc(X)_v$ is a subgroup of $\A_{n+2}$ with a $2$-transitive permutation representation of degree $n+2$, by the classification of $2$-transitive groups, the candidates for $(n+2,\Soc(X)_v )$ are
\[
(7,\PSL_3(2)), (8,\ASL_3(2)),(9,\PSL_2(8)),(9,\PGaL_2(8)),
\]
\[(9,\AGaL_1(9) \cap \A_9),(10,\PSL_2(9) ), (10,\M_{10} ).
\]
For these candidates, computation in \magma~\cite{Magma} shows that such a $(\Soc(X),2)$-arc-transitive graph $\Ga_M$ does not exist except when $(n+2,\Soc(X)_v)=(9,\AGaL_1(9) \cap \A_9)$. However, for this exception, we have $(X,X_v)=(\A_9,\AGaL_1(9) \cap \A_9)$ or $(\Sy_9,\AGaL_1(9))$, which yields that $|M|=|X_v|/|X:Y|=1$, contradicting our assumption.  
 
Therefore, $n\geq 9$. Since $X_v$ is has a $2$-transitive permutation representation of degree $n+2$, we derive from~\eqref{eq:orderofM} and Lemma~\ref{lm:order2trans} that $|M|< 2^{n-2}$. Let $B=MH=M{:}H$. Recall that $|M|$ is a proper divisor of $|H|$.
Then $M$ has no section isomorphic to $H=\A_n$, and so Lemma~\ref{lm:miniM} implies that $B=M \times H$.
Let $L$ be a normal subgroup of $G$ containing $M$ such that $L/M=\Soc(X)\cong\A_{n+2}$, and let $C=\Cen_L(M)$. 
Then $H \leq C$ and $C\cap M \leq \Z(C)$. Since $CM/M \unlhd L/M \cong \A_{n+2}$  and  $CM/M \geq HM/M\cong  \A_n $, we have $CM/M=L/M$. 
Hence
\[
\Z(C)/(C\cap M) \unlhd C/(C\cap M) \cong CM/M=L/M \cong \A_{n+2}. 
\]
Since $\Z(C)/(C\cap M) $ is abelian, this implies that $\Z(C)/(C\cap M) =1$. Accordingly, 
\[
C/\Z(C)=C/(C\cap M)\cong \A_{n+2}, 
\]
and so
\[
C/\Z(C)=(C/\Z(C))'=C'\Z(C)/\Z(C)\cong C^{'}/(C'\cap \Z(C))=C'/\Z(C').
\]
It follows that $C'/\Z(C')\cong\A_n$ and $C=C'\Z(C)$, which yields $C'=(C'\Z(C))'=C''$. Hence $C'$ is a covering group of $\A_{n+2}$.

Suppose $\Z(C')\not=1$. Then as the Schur multiplier of $\A_n$ is $\ZZ_2$ (see for instance \cite[Theorem 5.1.4]{K-Lie}), we have $\Z(C')=\ZZ_2$ and $C'\cong 2.\A_n$. Since $H \leq C$ and $C/C'$ is abelian, we have $H\leq C'$. Thus $C'$ has a subgroup $H\times \Z(C') \cong \A_{n} \times \ZZ_2$. However, this is not possible by \cite[Proposition 2.6]{DFZ17}. 

Therefore, $\Z(C')=1$ and so $C'\cong \A_{n+2}$. 
Since $C'$ is characteristic in $C$ and $C$ is normal in $L$, the group $C'$ is normal in $L$, and so $L\geq M \times C'$. 
This together with $L/M\cong\A_{n+2}\cong C'$ gives $L=M\times C'$. 
Since $|M|< 2^{n-2}<(n+2)!/2=|C'|$, the group $C'$ is characteristic in $L$ and hence normal in $G$. 
Since $H/(H \cap C') \cong HC'/C' \leq C/C'$ is abelian and $H=\A_n$, we derive that $H=H\cap C'$. 
Thus $C'\geq H$ is transitive on $V(\Ga)$. Let $N=C'$. 
Then $N=HN_\a$ with $N\cong\A_{n+2}$ and $H=\A_n$, and 
\[
N_\a=\frac{|N|}{|V(\Ga)|}=\frac{|N|}{|H|}=\frac{|\A_{n+2}|}{|\A_n|}=(n+2)(n+1).
\] 
This implies that $N_\a$ is a sharply $2$-transitive subgroup of $\A_{n+2}$. 
By~\cite[Chapter XII,~Theorem 9.1]{Huppert82b} we then have $\Soc(N_\a)=\ZZ_p^d$ and $n+2=p^d$ for some prime $p$ and integer $d$. 
Since $\Soc(N_\a)$ is characteristic in $N_\a$ and $N_\a \unlhd L_\a \cong (L/M)_v$, the group $(L/M)_v$ has a normal subgroup $\ZZ_p^d $. 
Recall that $\Ga_M$ is $(L/M,2)$-arc-transitive, $(L/M)_v$ is a $2$-transitive subgroup of $L/M =\A_{n+2}$, and $(L/M)_v$ acts faithfully on $\Ga_M(v)$. Thereby we conclude that $\Soc((L/M)_v)\cong \ZZ_p^d$, $\val(\Ga)=\val(\Ga_M)=p^d=n+2$, the action of $N_\a $ on $\Ga(\a)$ is faithful, and $N_\a^{\Ga(\a)}$ is a sharply $2$-transitive permutation group of degree $p^d$. 
As a consequence, $\Ga$ is $(N,2)$-arc-transitive, and part~(b) of Theorem~\ref{th:1} holds.
\end{proof}

\section{Proof of Theorem~\ref{Thm2}}\label{sec3}


Let $q$, $G$, $K$ and $\Ga$ be as in Construction~\ref{Exa2}, and let $M=\AGL_1(q^2)$. Then $K=M{:}\l\tau\r$, and the stabilizer $M_0$ of $0\in\bbF_{q^2}$ in $M$ is $\GL_1(q^2)$ and hence generated by a single element $\omega\in M_0$. Since both $g\tau$ and $\tau g$ fix $0$ and send $v$ to $v^{-q}$ for $v\in\bbF_{q^2}^\times$, we have $g\tau=\tau g$. Moreover, $\omega^g=\omega^{-1}$, and so $g\in\Nor_G(M_0)$. Thus $g$ normalizes $M_0{:}\l\tau\r=K_0$, the stabilizer of $0\in\bbF_{q^2}$ in $K$.

Let $N=\Nor_G(K)$. Then $\Soc(K)=\Soc(M)\cong\mathbb{F}_{q^2}^+$ is a normal subgroup of $N$, and $N$ is $2$-transitive on $\bbF_{q^2}$ as $K\geq M$ is $2$-transitive on $\bbF_{q^2}$. By the classification of $2$-transitive groups, this implies that $N\le\AGaL_1(q^2)$. Then since $\AGaL_1(q^2)\le\Nor_G(K)=N$, we obtain $N=\AGaL_1(q^2)$. Similarly, we have $\Nor_G(M)=\AGaL_1(q^2)=N$. Write $q=p^f$ with prime $p$ and integer $f$.

Suppose that $g\in N$. Then as $g$ fixes both $0$ and $1$, we have $g\in N_{0,1}=\mathrm{Gal}(\bbF_{q^2}/\bbF_p)$. Since $\mathrm{Gal}(\bbF_{q^2}/\bbF_p)$ is cyclic and $g$ has order $2$, the same order as $\tau\in\mathrm{Gal}(\bbF_{q^2}/\bbF_p)$, it follows that $g=\tau$. However, this leads to $\omega^{-1}=\omega^g=\omega^\tau=\omega^q$, which implies that $\omega$ has order dividing $q+1$, contradicting $\l\omega\r=M_0=\GL_1(q^2)$.

Thus we conclude that $g\notin N$. As a consequence, $K^g\neq K$, and so $K\cap K^g<K$. 
From $g\in\Nor_G(K_0)$ we deduce that $K\cap K^g\ge K_0$. 
Then since $K_0$ is maximal in $K$, we conclude that $K\cap K^g=K_0$. 
Hence the action of $K$ on $[K:K\cap K^g]$ is permutation isomorphic to the natural action of $K$ on $\mathbb{F}_{q^2}$ as a subgroup of $\AGaL_1(q^2)$, which is $2$-transitive. In particular, $\Ga$ has valency $|K:K\cap K^g|=|K:K_0|=q^2$.

Since $\tau$ is an involution with $\mathrm{Fix}(\tau)=\bbF_q$, it is a product of $(q^2-q)/2$ transpositions. 
From the condition $q\equiv3\pmod{4}$ we derive that $(q^2-q)/2$ is odd. 
Thus $\tau\notin\Alt(\bbF_{q^2})$, and so $|K\cap\Alt(\bbF_{q^2})|=|K|/2=q^2(q^2-1)$.
Since $\omega$ is a $(q^2-1)$-cycle with $q$ odd, we have $\omega\notin\Alt(\bbF_{q^2})$. 
Hence $K\cap\Alt(\bbF_{q^2})\geq\Soc(K){:}\l\omega^2,\omega\tau\r$, which in conjunction with $|K\cap\Alt(\bbF_{q^2})|=q^2(q^2-1)=|\Soc(K){:}\l\omega^2,\omega\tau\r|$ implies that 
\[
K\cap\Alt(\bbF_{q^2})=\Soc(K){:}\l\omega^2,\omega\tau\r.
\]
Since $1^\omega$ is a generator of $\mathbb{F}_{q^2}^\times$ and $q$ is odd, the orbits of $1^\omega$ and $(1^\omega)^q$ under $\l\omega^2\r$, respectively, are the orbits of $\l\omega^2\r$ on $\mathbb{F}_{q^2}^\times$.
Then since $(1^\omega)^q=1^{\omega\tau}$, we obtain $\mathbb{F}_{q^2}^\times=1^{\l\omega^2\r}\cup(1^{\omega\tau})^{\l\omega^2\r}$, which implies that $\l\omega^2,\omega\tau\r$ is transitive on $\mathbb{F}_{q^2}^\times$.
Therefore, the action of $\Soc(K){:}\l\omega^2,\omega\tau\r$ on $\mathbb{F}_{q^2}$ is $2$-transitive, and so is the action of $K\cap\Alt(\bbF_{q^2})$ on $[K:K\cap K^g]$.

Since $G=K\Alt(\bbF_{q^2})$, the right multiplication action of $\Alt(\bbF_{q^2})$ on $[G:K]$ is transitive. 
This together with the $2$-transitivity of $K\cap\Alt(\bbF_{q^2})$ on $[K:K\cap K^g]$ implies that $\Ga$ is $(\Alt(\bbF_{q^2}),2)$-arc-transitive.
Since $K_{0,1}=\l\tau\r$, we have 
\[
K\cap\Alt(\bbF_{q^2})_{0,1}=\l\tau\r\cap\Alt(\bbF_{q^2})_{0,1}=1.
\] 
Observe that $|K||\Alt(\bbF_{q^2})_{0,1}|=2q^2(q^2-1)|\A_{q^2-2}|=|G|$. 
We then conclude that $\Alt(\bbF_{q^2})_{0,1}$ acts regularly on $[G:K]$, and so $\Ga$ is a Cayley graph on $\Alt(\bbF_{q^2})_{0,1}\cong\A_{q^2-2}$.

Now we prove that $\Ga$ is connected, that is, $\l K,g\r=G$. Let $X=\l K,g\r$. Then $X$ is $2$-transitive on $\bbF_{q^2}$. Since $\tau\notin\Alt(\bbF_{q^2})$ and $\tau\in K\le X$, the group $X$ is not contained in $\Alt(\bbF_{q^2})$. Then by the classification of $2$-transitive groups, either $X=\Sym(\bbF_{q^2})$, or $X\le Y:=\AGL_{2f}(p)$. Suppose for a contradiction that the latter occurs. Then $g\in X_0\le Y_0$ and $K_0\le X_0\le Y_0$. As $g$ normalizes $M_0$, it follows that $g\in\Nor_{Y_0}(M_0)$. However, since $M_0$ is a Single cycle in $Y_0=\GL_{2f}(p)$, we see from~\cite[Theorem~7.3]{Huppert1967} that $\Nor_{Y_0}(M_0)=\GaL_1(q^2)=N_0$, and so $g\in N_0$, contradicting the conclusion that $g\notin N$.
Thus $X=\Sym(\bbF_{q^2})=G$, which implies that $\Ga$ is connected. This completes the proof of Theorem~\ref{Thm2}.

\section*{Acknowledgements}
This work was partially supported by NSFC 12061092.

\end{document}